\documentclass[a4paper,12pt,draft]{article}
 
\usepackage{amsmath,amsfonts,amssymb,amsthm}
\usepackage{geometry}
\usepackage{enumerate}
\usepackage[T1]{fontenc}
\usepackage[english]{babel}
\usepackage{dsfont,bm,bbm,mathrsfs}
\usepackage{stmaryrd}
\usepackage{url,color}
\usepackage{textcase}
\usepackage{mathtools}
\usepackage{natbib}
\usepackage{etoolbox} 

\usepackage{tikz}
\usetikzlibrary{intersections}

\allowdisplaybreaks

\def\cite{\citet}

\makeatletter
\def\@noindentfalse{\global\let\if@noindent\iffalse}
\def\@noindenttrue{\global\let\if@noindent\iftrue}
\def\@aftertheorem{%
  \@noindenttrue
  \everypar{%
    \if@noindent%
      \@noindentfalse\clubpenalty\@M\setbox\z@\lastbox%
    \else%
      \clubpenalty \@clubpenalty\everypar{}%
    \fi}}

\newtheorem{thm}{Theorem}[section]
\AfterEndEnvironment{thm}{\@aftertheorem}

\AfterEndEnvironment{prop}{\@aftertheorem}
\newtheorem{lma}[thm]{Lemma}
\AfterEndEnvironment{lma}{\@aftertheorem}
\newtheorem{cor}[thm]{Corollary}
\AfterEndEnvironment{cor}{\@aftertheorem}

\AfterEndEnvironment{conj}{\@aftertheorem}

\theoremstyle{definition}

\newtheorem{defi}[thm]{Definition}
\AfterEndEnvironment{defi}{\@aftertheorem}

\newtheorem{re}[thm]{Remark}
\AfterEndEnvironment{defi}{\@aftertheorem}
\newtheorem{ex}[thm]{Example}
\AfterEndEnvironment{ex}{\@aftertheorem}

\AfterEndEnvironment{pr}{\@aftertheorem}

\let\original@left\left
\let\original@right\right
\renewcommand{\left}{\mathopen{}\mathclose\bgroup\original@left}
\renewcommand{\right}{\aftergroup\egroup\original@right}

\renewcommand\section{\@startsection{section}{1}{\z@}%
{-3.5ex \@plus -1ex \@minus -.2ex}%
{1.3ex \@plus.2ex}%
{\center\small\sc\MakeTextUppercase}}

\def\subsection#1{\@startsection{subsection}{2}{0pt}%
{-3.5ex \@plus -1ex \@minus -.2ex}%
{1ex \@plus.2ex}%
{\bf\mathversion{bold}}{#1}}

\def\subsubsection#1{\@startsection{subsubsection}{3}{0pt}%
{\medskipamount}%
{-10pt}%
{\normalsize\itshape}{\kern-2.2ex. #1.}}

\def\be#1{\begin{equation*}#1\end{equation*}}
\def\ben#1{\begin{equation}#1\end{equation}}
\def\bes#1{\begin{equation*}\begin{split}#1\end{split}\end{equation*}}
\def\besn#1{\begin{equation}\begin{split}#1\end{split}\end{equation}}
\def\bea#1{\begin{align*}#1\end{align*}}
\def\bean#1{\begin{align}#1\end{align}}
\def\bg#1{\begin{gather*}#1\end{gather*}}

\def\bml#1{\begin{multline*}#1\end{multline*}}

\def\note#1{\par\smallskip%
\noindent\kern-0.01\hsize%
\setlength\fboxrule{0pt}\fbox{\setlength\fboxrule{0.5pt}\fbox{%
\llap{$\boldsymbol\Longrightarrow$ }%
\vtop{\hsize=0.98\hsize\parindent=0cm\small\rm #1}%
\rlap{$\enskip\,\boldsymbol\Longleftarrow$}
}}%
}

\usepackage{delimset}
\def\given{\mskip 0.5mu plus 0.25mu\vert\mskip 0.5mu plus 0.15mu}
\newcounter{bracketlevel}%
\def\@bracketfactory#1#2#3#4#5#6{%
\expandafter\def\csname#1\endcsname##1{%
\global\advance\c@bracketlevel 1\relax%
\global\expandafter\let\csname @middummy\alph{bracketlevel}\endcsname\given%
\global\def\given{\mskip#5\csname#4\endcsname\vert\mskip#6}\csname#4l\endcsname#2##1\csname#4r\endcsname#3%
\global\expandafter\let\expandafter\given\csname @middummy\alph{bracketlevel}\endcsname%
\global\advance\c@bracketlevel -1\relax%
}%
}
\def\bracketfactory#1#2#3{%
\@bracketfactory{#1}{#2}{#3}{relax}{0.5mu plus 0.25mu}{0.5mu plus 0.15mu}
\@bracketfactory{b#1}{#2}{#3}{big}{1mu plus 0.25mu minus 0.25mu}{0.6mu plus 0.15mu minus 0.15mu}
\@bracketfactory{bb#1}{#2}{#3}{Big}{2.4mu plus 0.8mu minus 0.8mu}{1.8mu plus 0.6mu minus 0.6mu}
\@bracketfactory{bbb#1}{#2}{#3}{bigg}{3.2mu plus 1mu minus 1mu}{2.4mu plus 0.75mu minus 0.75mu}
\@bracketfactory{bbbb#1}{#2}{#3}{Bigg}{4mu plus 1mu minus 1mu}{3mu plus 0.75mu minus 0.75mu}
}
\bracketfactory{clc}{\lbrace}{\rbrace}
\bracketfactory{clr}{(}{)}
\bracketfactory{cls}{[}{]}
\bracketfactory{abs}{\lvert}{\rvert}
\bracketfactory{norm}{\Vert}{\Vert}
\bracketfactory{floor}{\lfloor}{\rfloor}
\bracketfactory{ceil}{\lceil}{\rceil}
\bracketfactory{angle}{\langle}{\rangle}

\def\calF{\mathcal{F}}
\def\calG{\mathcal{G}}
\def\ahalf{{\textstyle\frac12}}
\def\eq#1{\eqref{#1}}

\def\Bi{\mathrm{Bi}}
\def\IP{\prob}

\renewcommand{\leq}{\leqslant}
\renewcommand{\geq}{\geqslant}
\renewcommand{\le}{\leq}
\renewcommand{\ge}{\geq}

\def\hKin{\hat{K}^{\mathrm{in}}}
\def\hKout{\hat{K}^{\mathrm{out}}}
\def\Kin{K^{\mathrm{in}}}
\def\Kout{K^{\mathrm{out}}}

\def\epsilon{\varepsilon}
\makeatother

\newcommand{\law}{\mathscr{L}}

\newcommand{\Pro}{\mathop{{}\mathbb{P}}}
\newcommand{\prob}{\Pro}
\renewcommand{\IP}{\Pro}
\newcommand{\E}{\mathop{{}\mathbb{E}}\mathopen{}}
\newcommand{\IE}{\E}
\newcommand{\mean}{\E}
\newcommand{\R}{\mathbb{R}}
\newcommand{\N}{\mathbb{N}}

\newcommand{\C}{\mathbb{C}}

\newcommand{\Var}{\mathop{\mathrm{Var}}\mathopen{}}
\newcommand{\var}{\Var}
\newcommand{\Cov}{\mathop{\mathrm{Cov}}\mathopen{}}
\newcommand{\cov}{\Cov}
\newcommand{\bigo}{\mathop{{}\mathrm{O}}\mathopen{}}
\newcommand{\dk}{d_{\mathrm{K}}}
\renewcommand{\min}{\mathop{\mathrm{min}}\mathopen{}}

\usepackage{subfig}

\newcommand\bone{\mathop{{}\mathbf{1}}\mathopen{}}

\newcommand{\bmu}{{\bm\mu}}
\newcommand{\bx}{{\bm x}}
\newcommand{\bX}{{\bm X}}

\newcommand{\cF}{{\cal F}}

\newcommand{\cb}{{\cal B}}
\newcommand{\cc}{{\cal C}}

\newcommand{\cf}{{\cal F}}

\newcommand{\cl}{{\cal L}}

\newcommand{\cn}{{\cal N}}
\newcommand{\cs}{{\cal S}}

\def\t#1{^{(#1)}}

\def\ignore#1{}

\newcounter{con}

\numberwithin{equation}{section}

\title{\sc\bf\large\MakeUppercase{Palm theory, random measures and Stein~couplings}}

\author{\sc Louis H. Y. Chen\footnote{{Postal address: Department of Mathematics,
National University of Singapore,
Block S17, 10 Lower Kent Ridge Road,
Singapore 119076,
Republic of Singapore. Email address: matchyl@nus.edu.sg}} \\ \it National University of Singapore \\ \\ \sc Adrian R\"ollin\footnote{{Department of Statistics and Applied Probability,
National University of Singapore,
6 Science Drive 2,
Singapore 117546, Republic of Singapore. e-mail: adrian.roellin@nus.edu.sg 
}} \\ \it National University of Singapore\\  \\ \sc Aihua Xia\footnote{{Postal address:
School
of Mathematics and Statistics, the University of Melbourne,
VIC 3010, Australia. Email address: aihuaxia@unimelb.edu.au. }} \\ \it University of Melbourne}

\def\parsedate #1:20#2#3#4#5#6#7#8\empty{20#2#3-#4#5-#6#7}
\def\moddate{\expandafter\parsedate\pdffilemoddate{\jobname.tex}\empty}
\date{30 August 2020}

\begin{document}
\maketitle

\begin{abstract}\noindent We establish a general Berry-Esseen type bound which gives optimal bounds in many situations under suitable moment assumptions. By combining the general bound with Palm theory, we deduce a new error bound for assessing the accuracy of normal approximation to statistics arising from random measures, including stochastic geometry. We illustrate the use of the bound in four examples: completely random measures, excursion random measure of a locally dependent random process, and the total edge length of Ginibre-Voronoi tessellations and of Poisson-Voronoi tessellations. Moreover, we apply the general bound to Stein couplings and discuss the special cases of local dependence and additive functionals in occupancy problems.
\end{abstract}



\section{Introduction} 

The pioneering work of \cite{Stein72}, well-known as Stein's method, provides a set of tools to estimate the error in the approximation of the distributions of random variables by a specific distribution, and it has proven to be particularly powerful in the presence of dependence. Indeed, many forms of Stein's method have been developed to study a variety of random phenomena, and the comprehensive monographs by \cite{BHJ} and \cite{CGS2011} give accounts to that diversity.

It has become clear over the past decades that Stein's method is naturally related to size biasing and its point process counterpart Palm theory; see, for example, the results of \cite{GR96}, \cite{CX2004} and \cite{GX06}. While \cite{GR96} and \cite{GX06} considered size-bias couplings, \cite{CX2004} studied Poisson process approximation for point processes using Palm theory.  The work of \cite{CX2004} suggests that for normal approximation for statistics resulting from a random measure including those in stochastic geometry, it may be fruitful to combine Stein's method with Palm theory as well. Thus, in this article, we study normal approximation for statistics associated with random measures through their Palm distributions.

To this end, we first prove a general result, Theorem~\ref{thm1}, which can be thought of as an extension of Theorem~2.1 of \cite{CS2004} to settings that are not restricted to local dependence. We then connect our result with Palm theory in Section~\ref{sec2} to bound the errors of normal approximation for statistics arising from random measures. In order to illustrate the approach, we then estimate in Section~\ref{sec3} the errors in the normal approximation for completely random measures, the excursion random measure of a locally dependent random process and the total edge length of Ginibre-Voronoi tessellations as well as Poisson-Voronoi tessellations. The first three examples do not assume the Poisson process as an underlying point process. Theorem~\ref{thm1} can also be easily combined with Stein couplings, giving rise to Theorem~\ref{thm5} in Section~\ref{sec4}, with applications to local dependence and problems from random occupancy.

Our main theorems are formulated in such a way so as to give optimal rates of convergence in many applications. The cost we have to pay are higher moment requirements, but in many applications these are naturally satisfied. 

\section{A General Theorem}

Let~$W$ be such that~$\E W=0$ and~$\Var W=1$. Theorem~2.1 of \cite{CS2004} shows that if~$W$ is a sum of LD1 (see Section~\ref{sec5} for more details) locally dependent random variables then there exists a random function~$\hat{K}(t)$ such that 
\ben{\label{1}
\E\clc{W f(W)} = \E\int_{-\infty}^{\infty} f'(W+t)\hat{K}(t)dt
}
for all absolutely continuous functions~$f$ for which the expectations exist. A bound on the Kolmogorov distance~$\dk(\law(W), \mathcal{N}(0,1))$ is then obtained without further dependence assumption. A crucial step in the proof is the use of a concentration inequality (Proposition~3.1 of \cite{CS2004}) established under the LD1 local dependence. A careful examination of the proof of the proposition reveals that the concentration inequality actually holds if~$W$ only satisfies~\eq{1}, in which case the bound is expressed in terms of~$\hat{K}(t)$ instead of the locally dependent random variables. Consequently, Theorem~2.1 of \cite{CS2004}  holds for any~$W$ if~$\Var(W)=1$ and there exists a random function~$\hat{K}(t)$ such that~\eq{1} holds. 

It was observed by \cite{CR2010} that the proof of Theorem~2.1 of \cite{CS2004} can be simplified if the concentration inequality is replaced by a recursive inequality, which was inspired by \cite{Raic2003} and which is \eq{14} in this paper. Using this approach, they obtain a bound for~$W$ satisfying a Stein coupling assumption. In this paper we use the recursive inequality approach to obtain a simpler bound for~$W$ assuming that~$\Var W = 1$ and that~\eq{1} holds for some random function~$\hat{K}(t)$. As in the proof of Theorem~2.1 of \cite{CS2004}, Young's inequality (\eq{11} in this paper) is used to separate the product of two random variables. A crucial step in the proof of \cite{CS2004} is to use Young's inequality together with the concentration inequality, whereas in this paper, it is used together with the recursive inequality. Also in this paper, the random function $\hat{K}(t)$ is decomposed as $\hKin(t)+\hKout(t)$ to allow greater flexibility in applications. We now state and prove the general theorem.

\begin{thm}\label{thm1}
Let~$W$ be such that~$\E W = 0$ and~$\Var W  = 1$. Suppose there is a random function~$\hat{K}(t)$ such that~\eq{1} holds for all absolutely continuous functions~$f$ for which the expectations exist, and assume we can write~$\hat{K}(t) = \hKin(t)+\hKout(t)$, where~$\hKin(t) = 0$ for~$|t|> 1$. Define~$K(t) = \E\hat{K}(t)$, $\Kin(t) = \E\hKin(t)$, and~$\Kout(t) = \E\hKout(t)$. Then
\ben{ \label{2}
\dk(\law(W),\mathcal{N}(0,1)) \le 2r_1+11r_2+ 5r_3 + 10r_4 + 7r_5,
}
where
\bea{
r_1 &= {\E\left|\int_{|t|\le1}\bclr{\hKin(t)-\Kin(t)}dt\right|},
&r_2 &= \int_{|t|\le 1}|t\Kin(t)|dt, \\
r_3 &= \E\int_{-\infty}^\infty\babs{\hKout(t)}dt,
&r_4 &= \E\int_{|t|\le 1}\bclr{\hKin(t)-\Kin(t)}^2dt, \\
r_5 &= \left(\E\int_{|t|\le 1}|t|\bclr{\hKin(t)-\Kin(t)}^2dt\right)^{1/2}.                 
}
\end{thm}

\begin{proof} 
From~\eq{1}, by letting~$f(w)=w$, we obtain~$\int_{-\infty}^{\infty}K(t)dt = 1$.
For~$x \in \R$ and~$\epsilon >0$, define
\be{
	h_{x,\epsilon}(w)\coloneqq \begin{cases}
		1&\text{if~$w\le x$,}\\
		0&\text{if~$w\ge x+\epsilon$,}\\
			1+\epsilon^{-1}(x-w)&\text{if~$x<w<x+\epsilon$.}
	\end{cases}
}
Let~$f_{x,\epsilon}$ be the bounded solution of the Stein equation
\ben{ \label{3}
	f_{x,\epsilon}'(w)-wf_{x,\epsilon}(w)=h_{x,\epsilon}(w)-\E h_{x,\epsilon}(Z),
}
where~$Z \sim \mathcal{N}(0,1)$. The bounded solution~$f_{x,\epsilon}$ of~\eq{3} is unique and is given by
\be{ 
	f_{x,\epsilon}(w) = - e^{\frac{1}{2}w^2}\int_w^\infty e^{-\frac{1}{2}	t^2}[h_{x,\epsilon}(t)-\E h_{x,\epsilon}(Z)]dt
}
(see \cite[p.~15]{CGS2011}). We have for all~$w, v \in \R$,
\ben{\label{4}
	0\le f_{x,\epsilon}(w) \le 1,  
	\qquad 
	|f'_{x,\epsilon}(w)|\le1,
	\qquad  
	|f'_{x,\epsilon}(w) - f'_{x,\epsilon}(v)| \le 1
}
and
\besn{ \label{5}
	|f'_{x,\epsilon}(w+t)-f'_{x,\epsilon}(w)| 
	&\le (|w|+1)|t|+ \frac{1}{\epsilon}\int_{t\wedge 0}^{t \vee 0}\bone[x \le w+u\le x+\epsilon]du \\ 	&\le (|w|+1)|t| +  \bone[x - 0\vee t \le w \le x - 0\wedge t + \epsilon].  
}
The bounds~\eq{4} and~\eq{5} were obtained by \cite[p.~2010]{CS2004}. Bounds for all cases of~$h$ and their proofs were given by \cite[Section~2.2]{CGS2011}. Now write
\besn{ \label{6} 
&\E h_{x,\epsilon}(W) - \E h_{x,\epsilon}(Z) \\ 
& \qquad = \E\int_{-\infty}^{\infty}f'_{x,\epsilon}(W)K(t)dt - \E\int_{-\infty}^{\infty}f'_{x,\epsilon}(W+t)\hat{K}(t)dt \\ 
& \qquad = \E\int_{|t|\le 1}f'_{x,\epsilon}(W)\bclr{\Kin(t)-\hKin(t)}dt \\ 
& \qquad \quad + \E\int_{-\infty}^\infty f'_{x,\epsilon}(W)\bclr{\Kout(t)-\hKout(t)}dt \\ 
& \qquad \quad + \E\int_{-\infty}^\infty\bclr{f'_{x,\epsilon}(W) - f'_{x,\epsilon}(W+t)}\hKout(t)dt \\ 
& \qquad \quad + \E\int_{|t|\le 1}\bclr{f'_{x,\epsilon}(W) - f'_{x,\epsilon}(W+t)}\bclr{\hKin(t) - \Kin(t)}dt \\ 
& \qquad \quad + \E\int_{|t| \le 1}\bclr{f'_{x,\epsilon}(W) - f'_{x,\epsilon}(W+t)}\Kin(t)dt \\
& \qquad \eqqcolon  R_1 + R_2 + R_3 + R_4 + R_5.
}
By~\eq{4}, we obtain bounds
\besn{ \label{7} 
|R_1| &= \left|\E\bbbclc{ f'_{x,\epsilon}(W)\int_{|t|\le 1}\bclr{\hKin(t)-\Kin(t)}dt}\right| \\ 
&\le \E\left|\int_{|t|\le 1}\bclr{\hKin(t)-\Kin(t)}dt\right| = r_1,
} 
\besn{\label{8} 
|R_2| &= \left|\E \bbbclc{f'_{x,\epsilon}(W)\int_{-\infty}^\infty\bclr{\Kout(t)-\hKout(t)}dt }\right| \\ 
&\le \E\int_{-\infty}^\infty \babs{\Kout(t)}dt + \E\int_{-\infty}^\infty|\hKout(t)|dt \\
&\le 2\E\int_{-\infty}^\infty\babs{\hKout(t)}dt = 2r_3,
}
and
\besn{ \label{9} 
|R_3|\le \E\int_{-\infty}^\infty\babs{\hKout(t)}dt = r_3.
}
By~\eq{5},
\besn{ \label{10} 
	|R_4| &\le \E\bbbclc{(|W|+1)\int_{|t|\le 1}|t|\babs{\hKin(t)-\Kin(t)}dt} \\ 
	&\quad + \E\int_{|t|\le 1}\bone[x - 0\vee t \le W \le x - 0 \wedge t + \epsilon]\babs{\hKin(t)-\Kin(t)}dt \qquad \\
	&\eqqcolon  R_{4,1} + R_{4,2}.
}
Recall Young's inequality: For~$a, b, c >0$, we have 
\ben{ \label{11}
ab \le \frac{ca^2}{2} + \frac{b^2}{2c}.
}
Using this inequality with~$c=\alpha>0$, $a=(|W|+1)\sqrt{|t|}$ and~$b=\sqrt{|t|}|\hKin(t)-\Kin(t)|$, we have
\bes{  
R_{4,1} &\le \frac{\alpha}{2}\E\int_{|t|\le 1}(|W|+1)^2|t|dt + \frac{1}{{2}\alpha}\E\int_{|t|\le 1}|t|\bclr{\hKin(t)-\Kin(t)}^2dt \\ 
&\le 2\alpha + \frac{1}{{2}\alpha}{\E}\int_{|t|\le 1}|t|\bclr{\hKin(t)-\Kin(t)}^2dt 
= 2\alpha + \frac{r_5^2}{2\alpha}.
}
By letting~$\alpha = r_5$,
\ben{\label{12}
R_{4,1} \le 2.5r_5.
}
Using the inequality~\eq{11} again, but with~$c=(2d+0.4|t|+0.4\epsilon)/(\theta\beta)$  for~$\theta, \beta > 0$,
$b=\bone[x-0\vee t \le W \le x - 0\wedge t + \epsilon]$ and~$a=|\hKin(t)-\Kin(t)|$, we obtain
\besn{\label{13} 
R_{4,2} &\le \frac{\theta\beta}{2}\E\int_{|t|\le 1}(2d+0.4|t|+0.4\epsilon)^{-1}\bone[x-0\vee t \le W \le x - 0\wedge t + \epsilon]dt \\ 
&\quad + \frac{1}{2\theta\beta}\E\int_{|t|\le 1}(2d+0.4|t|+0.4\epsilon)\bclr{\hKin(t)-\Kin(t)}^2dt.
}
Let
\be{
	d = \dk(\law(W),\mathcal{N}(0,1)), 
	\qquad 
	d_{\epsilon} = \sup_{x\in \R}|\E h_{x,\epsilon}(W) - h_{x,\epsilon}(Z)|;
}
then it is not difficult to see that, for~$a \le b$,
\ben{ \label{14}
  \prob[a \le W \le b]\leq 2d+ \frac{1}{\sqrt{2\pi}}(b-a)\le 2d+0.4(b-a),
  \qquad
  d \le d_\epsilon + 0.4\epsilon.
}
By~\eq{14},
\be{
\prob[x-0\vee t \le W \le x - 0\wedge t + \epsilon]\le 2d+0.4|t|+0.4\epsilon.
}
Using this, \eq{13} yields
\bes{  
R_{4,2}&\le \frac{\theta\beta}{2}\int_{|t|\le 1}dt + \frac{2d+0.4\epsilon}{2\theta\beta}\E\int_{|t|\le 1}\bclr{\hKin(t)-\Kin(t)}^2dt \\ 
&\quad + \frac{0.4}{2\theta\beta}\E\int_{|t|\le 1}|t|\bclr{\hKin(t)-\Kin(t)}^2dt \\ 
&= \theta\beta + \frac{d+0.2\epsilon}{\theta\beta}r_4 + \frac{0.2}{\theta\beta}r_5^2.
}
By letting~$\beta = d + 0.2\epsilon + r_5$, we obtain
\besn{ \label{16}
R_{4,2} \le \theta(d + 0.2\epsilon + r_5) + \frac{1}{\theta}r_4 + \frac{0.2}{\theta}r_5 
 = \theta d + 0.2\theta\epsilon + \frac{1}{\theta}r_4 + \left(\theta + \frac{0.2}{\theta}\right)r_5.
}
By~\eq{5} again, we have
\besn{ \label{17}
	|R_5| &\le \E\int_{|t|\le 1}(|W|+1)|t\Kin(t)|dt \\ 
	&\quad + \frac{1}{\epsilon}\E\int_{|t| \le 1}\int_{0\wedge t}^{0\vee t}\bone[x \le W + u\le x + \epsilon]|\Kin(t)|dudt \\ 
	&\le 2\int_{|t|\le 1}|t\Kin(t)|dt + \frac{1}{\epsilon}\int_{|t|\le 1}\int_{0\wedge t}^{0\vee t}\prob[x \le W+u \le x + \epsilon]|\Kin(t)|dudt \\ 
	&\le 2\int_{|t|\le 1}|t\Kin(t)|dt + \frac{1}{\epsilon}\int_{|t|\le 1}(2d + 0.4\epsilon)|t\Kin(t)|dt \\
	&= 2r_2 + \frac{2d+0.4\epsilon}{\epsilon}r_2.
}
Letting~$\epsilon = \frac{1}{2}d$ and combining~\eq{6}, \eq{7}, \eq{8}, \eq{9}, \eq{10}, \eq{12}, \eq{16} and~\eq{17}, we obtain
\be{
d_\epsilon \le 1.1\theta d + r_1 + 6.4r_2 + {3}r_3 + \frac{1}{\theta}r_4 + \left(2.5 + {\theta} + \frac{0.2}{\theta}\right)r_5.
}
This, together with~\eq{14}, yields
\be{
 d \le 1.1\theta d + 0.2d + r_1 + 6.4r_2 + {3}r_3 + \frac{1}{\theta}r_4 + \left(2.5 + {\theta} + \frac{0.2}{\theta}\right)r_5,
}
which implies
\be{
d \le (0.8-1.1\theta)^{-1}\left\{ r_1 + 6.4r_2 + 3r_3 + \frac{1}{\theta}r_4 + \left(2.5 + \theta + \frac{0.2}{\theta}\right)r_5\right\}.
}
Letting $\theta = 0.18$, we obtain
\be{ 
d \le 2r_1 + 11r_2 + 5r_3 + 10r_4 + 7r_5,
}
 and this proves Theorem~\ref{thm1}.
\end{proof}

\begin{re} We have introduced~$\hKin$ and~$\hKout$ mainly to allow for truncation. Since we have kept the theorem general, different types of truncation are possible, and we will show this in various applications in this article.  
\end{re}

\begin{ex}
We will check the optimality of the bounds in Theorem~\ref{thm1} by taking~$W$ as a sum of independent random variables. Let~$\xi_1,\cdots, \xi_n$ be independent with~$\E\xi_i= 0$ and~$\Var(\xi_i)=\sigma_i^2$, $i=1,\cdots,n$. Define~$B^2=\sum_{i=1}^n\sigma_i^2$, $X_i=\xi_i/B$, $i=1,\cdots,n$, and~$W=\sum_{i=1}^n X_i$. Then~$\E W=0$ and~$\Var(W)=1$ and~$W$ satisfies the Stein identity~\eq{1} with
\be{
\hat{K}(t) = \sum_{i=1}^n X_i\bclr{\bone[-X_i < t \le 0]-\bone[-X_i > t >0]}.
}
Define
\bea{
	\hKin(t)&= \sum_{i=1}^n X_i\bone[|X_i|\le 1]\bclr{\bone[-X_i < t \le 0]-\bone[-X_i > t >0]},\\
\hKout(t)&= \sum_{i=1}^n X_i\bone[|X_i| > 1]\bclr{\bone[-X_i < t \le 0]-\bone[-X_i > t >0]}.
}
Clearly~$\hat{K}(t)=\hKin(t)+ \hKout(t)$ and~$\hKin(t)=0$ for~$|t|>1$. Straightforward calculations yield
\bea{
r_1 &\le\left(\sum_{i=1}^n \var(X_i^2\bone[|X_i|\le 1])\right)^{1/2} \le \frac{\sqrt{\strut\sum_{i=1}^n \E\clc{\xi_i^4 \bone[|\xi_i|\le B]}}}{B^2};\\
r_2 &\le {\frac12}\sum_{i=1}^n \E\bclc{|X_i|^3\bone[|X_i| \le 1]} = \frac{\sum_{i=1}^n \E\clc{|\xi_i|^3\bone[|\xi_i|\le B]}}{{2}B^3};\\
r_3 &\le \sum_{i=1}^n \E\bclc{X_i^2\bone[|X_i|>1]}= \frac{\sum_{i=1}^n \E\clc{\xi_i^2\bone[|\xi_i|>B]}}{B^2}; \\
r_4 &\le \sum_{i=1}^n \E\bclc{|X_i|^3\bone[|X_i| \le 1]} = \frac{\sum_{i=1}^n \E\clc{|\xi_i|^3\bone[|\xi_i|\le B]}}{B^3};\\
r_5 &\le {\frac1{\sqrt{2}}}\left(\sum_{i=1}^n \E\bclc{X_i^4\bone[|X_i|\le 1]}\right)^{1/2} = \frac{\sqrt{ \strut\sum_{i=1}^n\E\clc{\xi_i^4 \bone[|\xi_i|\le B]}}}{{\sqrt{2}}B^2}.}
Assuming that~$\E|\xi_i|^3 < \infty$ for~$i=1,\dots,n$, we obtain
\ben{\label{18}
\dk(\mathcal{L}(W),\mathcal{N}(0,1))\le \frac{{7}\sqrt{\strut\sum_{i=1}^n \E\clc{\xi_i^4 \bone[|\xi_i|\le B]}}}{B^2}+ \frac{{15.5}\sum_{i=1}^n \E|\xi_i|^3}{B^3}.
}
If both~$\sum_{i=1}^n \E\clc{\xi_i^4 \bone[|\xi_i|\le B]}$ and~$\sum_{i=1}^n \E|\xi_i|^3$ are~$\bigo(B^2)$, such as in the i.i.d. case, then the bound in~\eq{18} is~$\bigo(B^{-1})$, which agrees with the order of the Berry-Esseen bound~$C\sum_{i=1}^n \E|\xi_i|^3/B^3$. We expect that in most applications, the bound on the Kolmogorov distance in~\eq{2} should give the optimal or near optimal order. In particular, for integer-valued random variables, the bounds can never be better than the scaling factor (see the general argument of \cite{Englund81}), so that, for instance, the bounds in Corollary \ref{cor10} on occupancy problems are optimal whenever the additive functional is integer-valued.
\end{ex}

\section{Random Measures}\label{sec2}

Let~$\Gamma$ be a locally compact separable metric space. Let~$\Xi$ be a random measure on~$\Gamma$ with finite intensity measure~$\Lambda$, and let~$\Xi_\alpha$ be the Palm measure associated with~$\Xi$ at~$\alpha \in \Gamma$ (see \cite[pp.~83, 103]{Kallenberg83}). We have 
\ben{\label{19}
\E\left\{\int_\Gamma f(\alpha,\Xi)\Xi(d\alpha)\right\}= \E\left\{\int_{\Gamma}f(\alpha,\Xi_\alpha)\Lambda(d\alpha)\right\}
}
for real-valued functions~$f(\cdot,\cdot)$ for which the expectations exist (see \cite[p.~84]{Kallenberg83}). If~$\Xi$ is a simple point process, the distribution of~$\Xi_\alpha$ can be interpreted as the conditional distribution of~$\Xi$ given that a point of~$\Xi$ at~$\alpha$ has occurred. On the other hand, if~$\Lambda(\{\alpha\})>0$, then~$\Xi(\{\alpha\})$ is a non-negative random variable with positive mean and~$\Xi_\alpha(\{\alpha\})$ is a~$\Xi(\{\alpha\})$-size-biased random variable. Therefore, in general, we may interpret the Palm measure as a ``size-biased random measure''.
For the special case where~$f$ is absolutely continuous from~$\R$ to~$\R$, we obtain
\ben{\label{20}
\E\clc{|\Xi|f(|\Xi|)}= \E\int_\Gamma f(|\Xi_\alpha|)\Lambda(d\alpha),
} 
provided the expectations and integral exist, where~$|\Xi|=\Xi(\Gamma)$.
Let~$\lambda = \Lambda(\Gamma) = \E|\Xi|$, $B^2 = \Var(|\Xi|)$ and define
\ben{ \label{21}
	W = \frac{|\Xi|-\lambda}{B},
	\qquad
 	W_\alpha = \frac{|\Xi_\alpha|- \lambda}{B}. 
}
Assume that~$\Xi$ and~$\Xi_\alpha$, $\alpha \in \Gamma$, are defined on the same probability space, and define
\be{
	\Delta_\alpha = W_\alpha - W,
	\qquad
	Y_\alpha = |\Xi_\alpha|-|\Xi|.
} 
From~\eq{20}, we have
\besn{  \label{22} 
&\E\clc{Wf(W)}\\
&\qquad= \frac{1}{B}\E\int_{\Gamma}(f(W_\alpha)-f(W))\Lambda(d\alpha)\\ 
&\qquad= \frac{1}{B}\E\int_{\Gamma}\int_0^{\Delta_\alpha}f'(W+t)dt\Lambda(d\alpha) \\ 
&\qquad= \frac{1}{B}\E\int_\Gamma\int_{-\infty}^{\infty}f'(W+t)\bclr{\bone[\Delta_\alpha > t >0] - \bone[\Delta_\alpha < t \le 0]}dt\Lambda(d\alpha) \\  
&\qquad= \E\int_{-\infty}^{\infty}f'(W+t)\hat{K}(t)dt,
}
where
\ben{ \label{23}
\hat{K}(t) = \frac{1}{B}\int_{\Gamma}\bclr{\bone[\Delta_\alpha > t >0] - \bone[\Delta_\alpha < t \le 0]}\Lambda(d\alpha).
}
We now apply Theorem~\ref{thm1} to~\eq{22} to obtain the following theorem.
\begin{thm} \label{thm2}
Let\/ $W$ and\/ $W_\alpha$, $\alpha \in \Gamma$, be as defined in~\eq{21}, and assume that\/ $\Xi$ and\/ $\Xi_\alpha$ are defined on the same probability space. Define~$\hat{K}(t)$ as in~\eq{23}, and let
\bea{
\hKin(t) &= \frac{1}{B}\int_{\Gamma}\bclr{\bone[\Delta_\alpha > t >0] - \bone[\Delta_\alpha < t \le 0]}\bone[|\Delta_\alpha|\le 1]\Lambda(d\alpha),\\
\hKout(t) &= \frac{1}{B}\int_{\Gamma}\bclr{\bone[\Delta_\alpha > t >0] - \bone[\Delta_\alpha < t \le 0]}\bone[|\Delta_\alpha| > 1]\Lambda(d\alpha).
}
Moreover, let
\be{
	K(t)=\E\hat{K}(t), 
	\qquad 
	\Kin(t)=\E\hKin(t), 
	\qquad
	\Kout(t)=\E\hKout(t).
}
Then
\be{
\dk(\law(W),\mathcal{N}(0,1)) \le 2r_1'+ 5.5r_2'+ 5r_3'+ 10r_4' + 7r_5',
}
where $r_1'$, $r_2'$, $r_3'$, $r_4'$ and~$r_5'$ are given by~\eq{24}, \eq{25}, \eq{26}, \eq{27} and~\eq{28} respectively.
\end{thm}
\begin{proof}
The proof of this theorem is reduced to calculating the error terms in Theorem~\ref{thm1}, which yields
\bean{ 
\begin{split}\label{24} 
r_1'\coloneqq{}& r_1 \\ 
={}& \frac{1}{B}{\mean\left|\int_\Gamma\bclr{\Delta_\alpha \bone[|\Delta_\alpha|\le 1]-\mean\{\Delta_\alpha \bone[|\Delta_\alpha|\le 1]\}}\Lambda(d\alpha)\right|}\\
={}& \frac{1}{B^2}{\mean\left|\int_\Gamma\bclr{Y_\alpha \bone[|Y_\alpha|\le B]-\mean\{Y_\alpha \bone[|Y_\alpha|\le B]\}}\Lambda(d\alpha)\right|};
\end{split}\\
\begin{split} \label{25}  
 2r_2 ={}&  \frac{1}{B}\int_\Gamma\E\bclc{\Delta_\alpha^2\bone[|\Delta_\alpha|\le 1]}\Lambda(d\alpha) \\
	={}& \frac{1}{B^3}\int_\Gamma\E\bclc{Y_\alpha^2\bone[|Y_\alpha|\le B]}\Lambda(d\alpha){\eqqcolon r_2'}; 
\end{split}\\
\begin{split} \label{26} 
 r_3 ={}& \frac{1}{B}\int_\Gamma \E\clc{|\Delta_\alpha|\bone[|\Delta_\alpha| > 1]}\Lambda(d\alpha) \\ 
={}& \frac{1}{B^2}\int_\Gamma \E\clc{|Y_\alpha|\bone[|Y_\alpha|>B]}\Lambda(d\alpha)\eqqcolon r_3';
\end{split}\\
\begin{split} \label{27} 
r_4'\coloneqq{}&  r_4 \\  
={}&\frac{1}{B^2}\int_{|t|\le 1}\int_\Gamma\int_\Gamma\Cov\bigl(\bone[1 \ge \Delta_\alpha > t >0] - \bone[-1 \le \Delta_\alpha < t \le 0],\\ 
&\kern10.5em \bone[1 \ge \Delta_\beta > t >0] - \bone[-1 \le\Delta_\beta < t \le 0]\bigr)\\
&\kern0.6\textwidth\times\Lambda(d\alpha)\Lambda(d\beta)dt \\ 
={}& \frac{1}{B^2}\int_0^1\int_\Gamma\int_\Gamma\Cov\bigl(\bone[1\ge\Delta_\alpha >t>0],\bone[1 \ge \Delta_\beta > t>0]\bigr)\\ 
& \kern0.6\textwidth  \times \Lambda(d\alpha)\Lambda(d\beta)dt \\  
& \quad + \frac{1}{B^2}\int_{-1}^0\int_\Gamma\int_\Gamma\Cov\bigl(\bone[-1\le\Delta_\alpha < t< 0],\bone[-1 \le \Delta_\beta < t<0]\bigr)\\ 
& \kern0.6\textwidth \times \Lambda(d\alpha)\Lambda(d\beta)dt;
\end{split}\\
\begin{split} \label{28} 
r_5'\coloneqq{}&  r_5 \\  
={}& \frac{1}{B}\bbbclr{\int_0^1\int_\Gamma\int_\Gamma t\Cov\bigl(\bone[1\ge\Delta_\alpha >t>0],\bone[1 \ge \Delta_\beta > t>0]\bigr)\\ 
& \kern0.6\textwidth  \times \Lambda(d\alpha)\Lambda(d\beta)dt \\ 
& \quad - \int_{-1}^0\int_\Gamma\int_\Gamma t\Cov\bigl(\bone[-1\le\Delta_\alpha < t< 0],\bone[-1 \le \Delta_\beta < t<0]\bigr)\\ 
& \kern0.6\textwidth \times \Lambda(d\alpha)\Lambda(d\beta)dt}^{1/2}.
\end{split}
}
This completes the proof of Theorem~\ref{thm2}.
\end{proof}

Using the fact that independence implies uncorrelatedness, we have the following corollary.

\begin{cor}\label{cor1}
Let~$\Xi$ be a random measure on~$\Gamma$ with finite mean measure~$\Lambda$ such that~$\E|\Xi|^4 < \infty$, and set~$B^2 = \Var(|\Xi|)$. Assume that~$\Xi$ and~$\Xi_\alpha$, $\alpha \in \Gamma$, are defined on the same probability space. Define 
\be{
W = \frac{|\Xi| - \E|\Xi|}{B},\qquad W_\alpha = \frac{|\Xi_\alpha| - \E|\Xi|}{B},\qquad \Delta_\alpha = W_\alpha - W. 
}
Assume that there is a set~$D \in \mathcal{B}(\Gamma\times\Gamma)$ such that $D$ is symmetric, i.e., $\{(x,y):\ (y,x)\in D\}=D$, and for all~$(\alpha,\beta) \neq D$, $\Delta_\alpha$ and~$\Delta_\beta$ are independent. Then
\be{
\dk(\mathcal{L}(W),\mathcal{N}(0,1)) \le 7s_1 + 5.5s_2 + 10s_3,
}
where
\bea{
s_1 &= \frac{1}{B^2}\left(\int_{(\alpha,\beta)\in D}\E\bclc{Y_\alpha^2 \bone[|Y_\alpha|\le B]}\Lambda(d\alpha)\Lambda(d\beta)\right)^\frac{1}{2};\\
s_2 &= \frac{1}{B^3}\int_\Gamma\E Y_\alpha^2\Lambda(d\alpha); \\
s_3 &= \frac{1}{B^3}\int_{(\alpha,\beta)\in D} \E \bclc{|Y_\alpha|\bone[|Y_\alpha|\le B]}\Lambda(d\alpha)\Lambda(d\beta).
}
\end{cor}

\begin{proof} By Theorem~\ref{thm2}, we have
\besn{\label{29}
r_1' &\le \frac{1}{B}\left(\mean\left\{\left[\int_\Gamma[\Delta_\alpha \bone[|\Delta_\alpha|\le 1]-\mean\clc{\Delta_\alpha \bone[|\Delta_\alpha|\le 1]}]\Lambda(d\alpha)\right]^2\right\}\right)^{1/2}\\
&= \frac{1}{B}\left(\int_{(\alpha,\beta)\in D}\Cov(\Delta_\alpha \bone[|\Delta_\alpha|\le 1],\Delta_\beta \bone[|\Delta_\beta|\le 1])\Lambda(d\alpha)\Lambda(d\beta)\right)^\frac{1}{2}\\
&\le \frac{1}{B}\left(\int_{(\alpha,\beta)\in D}\frac12\left(\var(\Delta_\alpha \bone[|\Delta_\alpha|\le 1])+\var(\Delta_\beta \bone[|\Delta_\beta|\le 1]\right) \Lambda(d\alpha)\Lambda(d\beta)\right)^\frac{1}{2}\\
&= \frac{1}{B}\left(\int_{(\alpha,\beta)\in D}\var(\Delta_\alpha \bone[|\Delta_\alpha|\le 1])\Lambda(d\alpha)\Lambda(d\beta)\right)^\frac{1}{2}\\
&\le \frac{1}{B^2}\left(\int_{(\alpha,\beta)\in D}\E\bclc{Y_\alpha^2 \bone[|Y_\alpha|\le B]}\Lambda(d\alpha)\Lambda(d\beta)\right)^\frac{1}{2}\eqqcolon s_1,
}
where the second equality is due to the symmetry of the set~$D$. Next,
\bean{
\begin{split} \notag
r_2'+r_3'&\le \frac1{B^3}\int_\Gamma\E Y_\alpha^2\Lambda(d\alpha)\eqqcolon s_2,
\end{split} \\
\begin{split} \notag
{r_4'} &= \frac{1}{B^2}\int_0^1\int_{(\alpha,\beta)\in D}\Cov\bigl(\bone[1\ge\Delta_\alpha >t>0],\bone[1 \ge \Delta_\beta > t>0]\bigr)\\ 
& \kern0.6\textwidth  \times \Lambda(d\alpha)\Lambda(d\beta)dt \\  
& \quad+ \frac{1}{B^2}\int_{-1}^0\int_{(\alpha,\beta)\in D}\Cov\bigl(\bone[-1\le\Delta_\alpha < t< 0],\bone[-1 \le \Delta_\beta < t<0]\bigr)\\  
& \kern0.6\textwidth  \times \Lambda(d\alpha)\Lambda(d\beta)dt \\
&\le \frac{1}{B^2}\int_{(\alpha,\beta)\in D} \E\clc{\min(|\Delta_\alpha|\bone[|\Delta_\alpha|\le 1],|\Delta_\beta|\bone[|\Delta_\beta|\le 1])}\Lambda(d\alpha)\Lambda(d\beta)\\
&\le {\frac{1}{B^3}\int_{(\alpha,\beta)\in D} \E\clc{|Y_\alpha|\bone[|Y_\alpha|\le B]}\Lambda(d\alpha)\Lambda(d\beta)\eqqcolon s_3};
\end{split}\\
\begin{split} \label{30}
{r_5'} &= \frac{1}{B}\bbbclr{\int_0^1\int_{(\alpha,\beta) \in D} t\Cov\bigl(\bone[1\ge\Delta_\alpha >t>0],\bone[1 \ge \Delta_\beta > t>0]\bigr)\\ 
& \kern0.6\textwidth  \times \Lambda(d\alpha)\Lambda(d\beta)dt \\ 
& \quad- \int_{-1}^0\int_{(\alpha,\beta)\in D} t\Cov\bigl(\bone[-1\le\Delta_\alpha < t< 0],\bone[-1 \le \Delta_\beta < t<0]\bigr)\\ 
&\kern0.6\textwidth  \times \Lambda(d\alpha)\Lambda(d\beta)dt}^{1/2} \\ 
&\le \frac{1}{B}\bbbclr{\int_0^1\int_{(\alpha,\beta) \in D} t\E\clc{ \bone[1\ge\Delta_\alpha >t>0]\bone[1 \ge \Delta_\beta > t>0]}\\ 
& \kern0.6\textwidth \times \Lambda(d\alpha)\Lambda(d\beta)dt \\ 
& \quad- \int_{-1}^0\int_{(\alpha,\beta)\in D} t\E\clc{ \bone[-1\le\Delta_\alpha < t< 0]\bone[-1 \le \Delta_\beta < t<0]}\\ 
& \kern0.6\textwidth  \times \Lambda(d\alpha)\Lambda(d\beta)dt}^{1/2} \\ 
&= \frac{1}{{\sqrt{2}}B}\bbbclr{\int_{(\alpha,\beta) \in D}\E\bclc{\min\bigl(\Delta_\alpha^2\bone[|\Delta_\alpha|\le 1],\Delta_\beta^2\bone[|\Delta_\beta|\le 1]\bigr)}\Lambda(d\alpha)\Lambda(\beta)}^\frac{1}{2} \\
&\le {\frac{1}{\sqrt{2}B^2}\bbbclr{\int_{(\alpha,\beta) \in D}\E \bclc{Y_\alpha^2\bone[|Y_\alpha|\le B]}\Lambda(d\alpha)\Lambda(\beta)}^\frac{1}{2}
=\frac1{\sqrt{2}}s_1}.
\end{split}
}
The proof of the corollary is completed by combining~\eq{29} to~\eq{30}.
\end{proof}

\section{Applications}\label{sec3}

\subsection{Completely random measures}

A random measure~$\Xi$ on the carrier space~$(\Gamma,\cb(\Gamma))$ is said to be \textit{completely random} (see~\cite{Kingman67}) if for any~$k\ge 1$ and any pairwise disjoint sets~$A_1,\dots,A_k\in\cb(\Gamma)$, $\Xi(A_i)$, $1\le i\le k$, are independent. Well-known examples include the compound Poisson process with cluster distributions on~$\R_+\coloneqq [0,\infty)$ (see  \cite[p.~198]{DV03}), the Gamma process (see \cite[p.~11]{DV08}) and the P\'olya sum process (see \cite{Zessin09} and \cite{Rafler11}). The former two processes cannot in general be represented as an integral of a random field with respect to a point process with finite mean measure, hence they are not covered by the general theory of \cite{BX06}. 

\begin{thm} \label{thm3}Let~$\Xi$ be a completely random measure with mean measure~$\Lambda$ and finite fourth moment~$\mean|\Xi|^4<\infty$. Let~$\mu\coloneqq \mu_\Xi\coloneqq \Lambda(\Gamma)$, $B^2\coloneqq \var(|\Xi|)$ and~$W=(|\Xi|-\mean|\Xi|)/B$, then
\bean{
&\dk({\law}(W),\cn(0,1))\notag\\
\begin{split}\label{31}
&\qquad\le \frac{10}{B^2}\left(\sum_{\alpha\in\Gamma}\E\bclc{\Xi(\{\alpha\})^3}\Lambda(\{\alpha\})\right)^{1/2}+\frac{5.5}{B^3}\E\sum_{\alpha\in\Gamma}\Xi(\{\alpha\})^3\\
&\qquad\quad+\frac{25.5}{B^3}\sum_{\alpha\in\Gamma}\E\bclc{\Xi(\{\alpha\})^2}\Lambda(\{\alpha\})
\end{split}\\
\begin{split}\label{32}
&\qquad\le \frac{10}{B^2}\left(\sum_{\alpha\in\Gamma}\E\bclc{\Xi(\{\alpha\})^3}\Lambda(\{\alpha\})\right)^{1/2}+\frac{31}{B^3}\E\sum_{\alpha\in\Gamma}\Xi(\{\alpha\})^3.
\end{split}
}
\end{thm}

\begin{re} (1) If~$\Lambda$ is diffuse at~$\alpha$ (i.e., $\Lambda(\{\alpha\})=0$), then~$\Xi(\{\alpha\})=0$ a.s. Hence, if~$\Lambda$ is a diffuse measure, then the bound~\eq{31} is reduced to
\be{
	\dk(\mathcal{L}(W),\mathcal{N}(0,1)) \le \frac{5.5}{B^3} \E\sum_{\alpha\in\Gamma}\Xi(\{\alpha\})^3
= \frac{5.5}{B^3}\int_\Gamma\E\bclc{\Xi_\alpha(\{\alpha\})^2}\Lambda(d\alpha).
}

\noindent(2) For a simple Poisson point process with $\Lambda(\Gamma)=\lambda$, the bound in \eq{32} becomes~$31\lambda^{-1/2}$, which compares favourably with those in the literature; see, for example, \cite{LSY17}.
\end{re}

\begin{proof}[Proof of Theorem~\ref{thm3}] Using similar notation as in Corollary~\ref{cor1}, we have~$Y_\alpha={\Xi_\alpha(\{\alpha\})-\Xi(\{\alpha\})}$ and~$Y_\alpha$ is independent of~$Y_\beta$ unless~$\alpha=\beta$. H\"older's inequality ensures that
\ben{
  \mean\bclc{\Xi(\{\alpha\})^2}\Lambda(\{\alpha\})\le \mean\bclc{\Xi(\{\alpha\})^3}.\label{33}
}
Hence, direct computation gives
\bes{
s_1&\le\frac1{B^2}\bbbclr{\sum_\Gamma\mean\bclc{Y_\alpha^2}\Lambda(\{\alpha\})^2}^{1/2}\\
&\le\frac1{B^2}\bbbclr{\sum_\Gamma\bclr{\mean\bclc{\Xi(\{\alpha\})^2}+\E\bclc{\Xi_\alpha(\{\alpha\})^2}}\Lambda(\{\alpha\})^2}^{1/2}\\
&=\frac{1}{B^2}\bbbclr{\sum_\Gamma\bclr{\E\bclc{\Xi(\{\alpha\})^2}\Lambda(\{\alpha\})^2+\E\bclc{\Xi(\{\alpha\})^3}\Lambda(\{\alpha\})}}^{1/2}\\
&\le\frac{\sqrt{2}}{B^2}\bbbclr{\sum_\Gamma\E\bclc{\Xi(\{\alpha\})^3}\Lambda(\{\alpha\})}^{1/2},
}
where the equality is due to~\eq{19} and the last inequality follows from \eq{33}. The same reasoning gives
\bes{
s_2&\le\frac1{B^3}\int_\Gamma\bclr{\mean\bclc{\Xi(\{\alpha\})^2}+\E\bclc{\Xi_\alpha(\{\alpha\})^2}}\Lambda(d\alpha)\\
&=\frac{1}{B^3} \bbbclr{\sum_{\alpha\in\Gamma}\E\bclc{\Xi(\{\alpha\})^2}\Lambda(\{\alpha\})+\E\int_\Gamma\Xi(\{\alpha\})^2\Xi(d\alpha)}\\
&=\frac{1}{B^3} \bbbclr{\sum_{\alpha\in\Gamma}\E\bclc{\Xi(\{\alpha\})^2}\Lambda(\{\alpha\})+\E\sum_{\alpha\in\Gamma}\Xi(\{\alpha\})^3},
}
where the first equality follows from the fact that~$\Lambda(\{\alpha\})=0$ implies~$\Xi(\{\alpha\})=0$ a.s. and~\eq{19}.
Finally,
\bes{
s_3&\le\frac1{B^3}\sum_\Gamma\bclr{\mean\clc{\Xi(\{\alpha\})}+\E\clc{\Xi_\alpha(\{\alpha\})}}\Lambda(\{\alpha\})^2\\
&=\frac1{B^3}\sum_\Gamma\bclr{\mean\clc{\Xi(\{\alpha\})}\Lambda(\{\alpha\})^2+\E\bclc{\Xi(\{\alpha\})^2}\Lambda(\{\alpha\})}\\
&\le\frac2{B^3}\sum_\Gamma\E\bclc{\Xi(\{\alpha\})^2}\Lambda(\{\alpha\}).
}
Combining these estimates and Corollary~\ref{cor1} gives~\eq{31}. \eq{32} is an immediate consequence of~\eq{31} and~\eq{33}.
\end{proof}

\begin{cor}Let~$\Xi^{(i)}$ for $1\le i\le n$ be independent random measures on the carrier space~$(S,\cs)$. Define\/ $\Xi=\sum_{i=1}^n\Xi^{(i)}$, $B^2=\var(|\Xi|)$ and $W=(|\Xi|-\mean|\Xi|)/B$. Then
\be{
 \dk({\law}(W),\cn(0,1))\le \frac{10}{B^2}\bbbclr{\sum_{i=1}^n\mean|\Xi^{(i)}|\mean|\Xi^{(i)}|^3}^{1/2}+\frac{31}{B^3}\sum_{i=1}^n\mean|\Xi^{(i)}|^3
}
\end{cor} 

\begin{proof} Define~$\Xi'=\sum_{i=1}^n|\Xi^{(i)}|\delta_i$, where~$\delta_i$ is the Dirac measure at~$i$, then~$\Xi'$ is a completely random measure on the carrier space~$\Gamma\coloneqq \{1,\dots,n\}$ with mean measure~$\Lambda'(\{i\})=\mean|\Xi^{(i)}|$, $i\in S'$. We have~$B^2=\var(|\Xi'|)$, $W=(|\Xi'|-\mean|\Xi'|)/B$, hence the claim follows from~\eq{32}.
\end{proof}

\subsection{Excursion random measure}

Let~$(S,\cb(S))$ be a metric space and~$\{X_t,\ 0\le t\le T\}$ be an~$S$-valued random process. Define~$\cf_{a,b}=\sigma\{X_t:\ a\le t\le b\}$, for~$0\le a\le b\le T$. We say that~$\{X_t,\ 0\le t\le T\}$ is~\emph{$l$-dependent} with~$l>0$ if~$\cf_{0,b}$ is independent of~$\cf_{b+l,T}$ for all~$0\le b<b+l\le T$. We define the excursion random measure
\be{
 \Xi(dt)=\bone_E(t,X_t)dt,\ E\in \cb([0,T]\times S).
}
Define~$\mu=\mean\Xi([0,T])$, $B=\sqrt{\var(|\Xi|)}$ and~$W=(|\Xi|-\mu)/B$.

The excursion random measure of a stationary process was defined by \cite{HL98}.
It was shown by \cite{HL98} that the asymptotic distribution of the
excursion random measure at high levels of exceedances gives a range of useful information about the extremal behavior of the stationary process. Under very general conditions, \cite{HL98} demonstrated that various asymptotic properties of the excursion random measures can be established. Our Theorem~\ref{thm2}
can be used to prove the following normal approximation error bound for the total excursion time of $l$-dependent random processes.

\begin{thm} For the~$l$-dependent process~$\{X_t,\ 0\le t\le T\}$, we have
\ben{\dk({\law}(W),\cn(0,1))\le \frac{({14}\sqrt{2}+8)l^{3/2}\mu^{1/2}}{B^2}+
\frac{{102}l^2\mu}{B^3}.\label{67}}
\end{thm}

\begin{proof} Write~$\Lambda(dt)=\mean\Xi(dt)$ and~$N_\alpha=[0,T]\cap[\alpha-l,\alpha+l]$ for~$\alpha\in[0,T]$. Since 
$\{X_t,\ 0\le t\le T\}$ is~$l$-dependent, we can take~$\Xi_\alpha$ such that~$Y_\alpha=\Xi_\alpha(N_\alpha)-\Xi(N_\alpha)$ for all~$\alpha\in[0,T]$, and~$Y_\alpha$ is independent of~$Y_\beta$ for all~$|\alpha-\beta|>2l$. Moreover, we have~$|Y_\alpha|\le 2l$. Hence,
\besn{\label{68}
{r_1'}&\leq\frac1{B^2}\left(\iint_{|\beta-\alpha|\le 2l}\cov\left(Y_\alpha \bone[|Y_\alpha|\le B],Y_\beta \bone[|Y_\beta|\le B]\right)\Lambda(d\alpha)\Lambda(d\beta)\right)^{1/2}\\
&\le\frac1{B^2}\left(\iint_{|\beta-\alpha|\le 2l}\frac12\left[\var\left(Y_\alpha \bone[|Y_\alpha|\le B]\right)+\var\left(Y_\beta \bone[|Y_\beta|\le B]\right)\right]\Lambda(d\alpha)\Lambda(d\beta)\right)^{1/2}\\
&=\frac1{B^2}\left(\iint_{|\beta-\alpha|\le 2l}\var\left(Y_\alpha \bone[|Y_\alpha|\le B]\right)\Lambda(d\alpha)\Lambda(d\beta)\right)^{1/2}\\
&\le\frac1{B^2}\left(\iint_{|\beta-\alpha|\le 2l}\mean\left(Y_\alpha^2\right)\Lambda(d\alpha)\Lambda(d\beta)\right)^{1/2}\leq\frac{4}{B^2}\sqrt{l^3\mu}.
}
Similarly,
\bea{
r_2'+r_3'&=\frac1{B^3}\int_0^T\mean (Y_\alpha^2)\Lambda(d\alpha)
\le4l^2\mu/B^3,\\
\begin{split}
r_4'&=\frac1{B^2}\int_0^1\iint_{|\beta-\alpha|\le 2l}\cov\left( \bone[1\ge \Delta_\alpha>t>0], \bone[1\ge \Delta_\beta>t>0]\right)\\
&\kern0.6\textwidth \times \Lambda(d\alpha)\Lambda(d\beta)dt\\
&\quad+\frac1{B^2}\int_{-1}^0\iint_{|\beta-\alpha|\le 2l}\cov\left( \bone[-1\le \Delta_\alpha<t<0], \bone[-1\le \Delta_\beta<t<0]\right)\\
&\kern0.6\textwidth \times\Lambda(d\alpha)\Lambda(d\beta)dt\\
&\le\frac1{B^2}\int_0^1\iint_{|\beta-\alpha|\le 2l}\mean \bone[1\ge \Delta_\alpha>t>0]\Lambda(d\alpha)\Lambda(d\beta)dt\\
&\quad+\frac1{B^2}\int_{-1}^0\iint_{|\beta-\alpha|\le 2l}\mean \bone[-1\le \Delta_\alpha<t<0]\Lambda(d\alpha)\Lambda(d\beta)dt\\
&=\frac1{B^3}\iint_{|\beta-\alpha|\le 2l}\mean|Y_\alpha|\Lambda(d\alpha)\Lambda(d\beta)\le\frac{8l^2}{B^3}\mu,
\end{split}
}
and
\besn{\label{69}
r_5'&=\frac1B\bbbclc{\int_0^1\iint_{|\beta-\alpha|\le 2l}t\cov\left( \bone[1\ge \Delta_\alpha>t>0], \bone[1\ge \Delta_\beta>t>0]\right)\\
&\kern0.6\textwidth \times\Lambda(d\alpha)\Lambda(d\beta)dt\\
&\quad-\int_{-1}^0\iint_{|\beta-\alpha|\le 2l}t\cov\left( \bone[-1\le \Delta_\alpha<t<0], \bone[-1\le \Delta_\beta<t<0]\right)\\
&\kern0.6\textwidth \times\Lambda(d\alpha)\Lambda(d\beta)dt}^{1/2}\\
&\le\frac1B\left\{\int_0^1\iint_{|\beta-\alpha|\le 2l}t\mean \bone[1\ge \Delta_\alpha>t>0]\Lambda(d\alpha)\Lambda(d\beta)dt\right.\\
&\quad-\left.\int_{-1}^0\iint_{|\beta-\alpha|\le 2l}t\mean \bone[-1\le \Delta_\alpha<t<0]\Lambda(d\alpha)\Lambda(d\beta)dt\right\}^{1/2}\\
&\le\frac1B\left\{\frac12\iint_{|\beta-\alpha|\le 2l}\mean\clc{\Delta_\alpha^2}\Lambda(d\alpha)\Lambda(d\beta)dt\right\}^{1/2}
\le\frac{2\sqrt{2}l^{3/2}\mu^{1/2}}{B^2}.
}
Finally, we have from Theorem~\ref{thm2} that
\be{
 \dk({\law}(W),\cn(0,1))\le 2r_1'+5.5(r_2'+r_3')+10r_4'+7r_5',
}
so collecting~\eq{68} to~\eq{69}, we obtain~\eq{67}.
\end{proof}

\begin{cor} Let~$I_i$, $1\le i\le n$, be independent indicator random variables such that~$\prob[I_i=1]=p_i$. Let~$S_n=\sum_{i=1}^{n-k+1}\prod_{j=i}^{i+k-1}I_j$, the number of~$k$-runs in the sequence. Define~$\mu_n=\mean S_n=\sum_{i=1}^{n-k+1}\prod_{j=i}^{i+k-1}p_j$, $B_n=\sqrt{\var(S_n)}$, $W=(S_n-\mu_n)/B_n$, then
\be{
	\dk({\law}(W),\cn(0,1))
		\le \frac{({14}\sqrt{2}+8)k^{3/2}\mu_n^{1/2}}{B_n^2}
				+ \frac{{102}k^2\mu_n}{B_n^3}.
}
In particular, if~$p_i=p\in(0,1)$ for all~$i$, then 
\be{
 \dk({\law}(W),\cn(0,1))=\bigo(n^{-1/2}).
}
\end{cor}

\begin{proof} We define~$X_t\coloneqq X_{\lfloor t\rfloor}=\prod_{j=\lfloor t\rfloor}^{\lfloor t\rfloor+k-1}I_j$, $0\le t< n-k+2$. Let~$E=[1,n-k+2)\times\{1\}$, $\Xi(dt)= \bone_E(t,X_t)dt$, then the claim follows from~\eq{67} with~$l=k$. \end{proof}

\subsection{The total edge length of Ginibre-Voronoi tessellations}

The Ginibre point process (see \cite{G65}, \cite{Sosh00}, \cite{Mehta91} and \cite{Goldman10}) has attracted a considerable attention recently because of its wide use in modeling mobile networks (see \cite{TL14}, \cite{MS14a,MS14b} and \cite{KRX16}). The Ginibre point process is a special class of the Gibbs point process family and it exhibits a repulsion between the points. The repulsive character makes the cells more regular than those coming from a Poisson point process, hence in applications the Ginibre-Voronoi tessellation often fits better than the Poisson-Voronoi tessellation (see \cite{Rider04}, \cite{LCH90} and \cite{Goldman10}). 

The Ginibre point process is defined through the factorial moment measures. For a locally finite process~${\bm Y}$ on a Polish space~$S$, the~$n$-th order \emph{factorial moment measure}~$\nu^{(n)}$ of~${\bm Y}$  is defined by the relation (see \cite[pp.~109--110]{Kallenberg83})

\bes{
& \mean \left[ \int_{S^n}f(y_1,\dots,y_n) {\bm Y}(dy_1)\left({\bm Y}-\delta_{y_1}\right)(dy_2)\ldots \left({\bm Y}-\sum_{i=1}^{n-1}\delta_{y_i}\right)(dy_n)  \right]\\
&\qquad=\int_{S^n}f(y_1,\dots,y_n)  \nu^{(n)}(dy_1,\dots,dy_n),
}
where~$f$ ranges over all Borel measurable functions~$h:\ S^n\rightarrow [0,\infty)$. The Ginibre point process on the complex plane~$\C$ is defined as follows.

\begin{defi}
We say the point process~$\bX$ on the complex plane~$\C$~$(\cong\R^2)$ is the Ginibre point process if its factorial moment measures are given by
\be{
\nu\t n (dx_1, \ldots, dx_n) = {\rho\t n(x_1,\ldots, x_n) dx_1\dots dx_n},\ n\ge 1,
}
where~$\rho\t n(x_1,\ldots, x_n)$ is the determinant of the~$n\times n$ matrix with~$(i,j)$th entry 
\be{
	K(x_i,x_j)
		=\frac 1\pi e^{-\frac{1}{2}(|x_i|^2+|x_j|^2)}
			e^{ x_i\bar x_j}.
}
Here and in the sequel, $\bar x$ and~$|x|$ are the complex conjugate and modulus of~$x$.
\end{defi}

The Ginibre point process has the mean measure~$\bmu(dx)=\frac1\pi dx$. \cite{Goldman10} stated that the Palm process~$\bX_x$ of the Ginibre point process~$\bX$ at the location~$x$ satisfies
\ben{\bX\stackrel{d}{=}(\bX_x\setminus\{x\})\cup\{x+Z\},\label{70}
}
where~$\stackrel{d}{=}$ stands for `equals in distribution', $Z=(Z_1,Z_2\sqrt{-1})$ with~$(Z_1,Z_2)$ having bivariate normal on~$\R^2$ with mean~$(0,0)$ and covariance matrix~$\begin{bmatrix}1/2&0\\ 0&1/2\end{bmatrix}$. That is, the Palm process~$\bX_x$ can be obtained by removing a point from the process which is Gaussian distributed from~$x$ and then adding~$x$ to~$\bX$. It is still an open problem 
to know how~$Z$ is correlated with~$\bX_x\setminus\{x\}$ (see \cite[Problem~2, p.~27]{Goldman10}).

As \cite{SY13} (see also \cite{XY15}), we consider the window~$Q_\lambda\coloneqq \{(s_1,s_2\sqrt{-1}):\ -0.5\sqrt{\lambda}\le s_1,s_2\le 0.5\sqrt{\lambda}\}\subset\C$. For a realization~$\bx$ of~$\bX$ and~$x\in\bx$, let 
$\cc(x,\bx)$ be the set of every point in~$\C$ whose (Euclidean) distance to~$x$ is less than or equal to its distance to any other point of~$\bx$. 
The set~$\cc(x,\bx)$ is called the \emph{Voronoi cell} centered at~$x$ and the collection of~$\cc(x,\bx)$, $x\in\bx$, is called the \emph{Voronoi tessellation} induced by~$\bx$. 

Note that when the Voronoi cell centers are close to the boundary of
$Q_\lambda$, our definition of Voronoi cells is slightly different from that of \cite{PY01} (see also \cite{BY05}, \cite{SY13} and \cite{XY15}). This is because the Voronoi cells defined by \cite{PY01} do not satisfy the translation invariant property which is a crucial condition for obtaining the central limit theorems of the Voronoi tessellation statistics.\footnote{This minor issue was noted by \cite{Penrose07}, and we thank J. Yukich for bringing this to our attention.}

If we define the random measure
\be{
 \Xi(dx)=L(x,\bX)\bX(dx),
}
where~$L(x,\bX)\coloneqq L_\lambda(x,\bX)$ is one half the total edge length of the \emph{finite} length edges (hence we exclude all infinite edges) in the cell~$\cc(x,\bX)$, then
the total edge length of the Ginibre-Voronoi tessellation induced by~$\bX$ with centers in~$\bX\cap Q_\lambda$ can be written as 
\be{
 {{\cl(\lambda)}}\coloneqq |\Xi|=\int_{Q_\lambda}L(x,\bX)\bX(dx).
}

\begin{thm}\label{thm4} Let~$B^2=\var({{\cl(\lambda)}})$ and~$W=({{\cl(\lambda)}}-\mean {{\cl(\lambda)}})/B$. We have
\ben{\lim_{\lambda\to\infty}\lambda^{-1} \mean {{\cl(\lambda)}}\in(0,\infty), \ \ \ 
\lim_{\lambda\to\infty}\lambda^{-1}B^2\in(0,\infty)\label{71}
} 
and
\ben{
\dk({\law}(W),\cn(0,1))=\bigo\left(\lambda^{-1/2}\ln\lambda\right).
\label{72}
}
\end{thm}

\begin{re} The Ginibre-Voronoi tessellation is a special case of the Gibbs-Voronoi tessellations
studied by \cite{XY15}. Theorem~2.3 of \cite{XY15} gives 
\be{
 \dk({\law}(W),\cn(0,1))=\bigo\left(\lambda^{-1/2}(\ln\lambda)^4\right),
}
which is slightly worse than~\eq{72}.
\end{re}

\begin{re}
The error estimate for the total edge length of the Ginibre-Voronoi tessellation is also valid for the more general class of~$\alpha$-Ginibre point processes with~$0<\alpha<1$. As a matter of fact, an~$\alpha$-Ginibre point process can be constructed by ``deleting,
independently and with probability~$1-\alpha$, each point of the Ginibre point process and then applying the homothety of ratio~$\sqrt{\alpha}$ to the remaining points in
order to restore the intensity of the process'' (\cite{Goldman10}). Hence, for~$\alpha$-Ginibre Voronoi tessellations,  except notational complexity, our proof goes through without any difficulty. 
\end{re}

\begin{re} We do not know if the bound of~\eq{72} is of the correct order. 
\end{re}

To prove Theorem~\ref{thm4}, we need the following lemmas. We note that the estimate of the void probability \eq{73}, albeit very simple, is not new and better estimates were given by \cite[Lemma~1.7 in Supplement]{BYY19}.

\begin{lma}\label{lem1} For~$A\subset \C$ with the area~$|A|$, we have
\ben{\prob[\bX(A)=0]\le e^{-|A|/\pi}\label{73}}
and 
\ben{ 
\mean[\bX(A)^l]\le 2^{\frac12l(l-1)}[1\vee(|A|/\pi)]^l\label{74}
}
for all~$l\in\N\coloneqq \{1,2,\dots\}$.
\end{lma}

\begin{proof} For~$\theta\in(0,1)$, let~$g(\theta)=\mean \bclc{(1-\theta)^{\bX(A)}}$, then 
\ben{
g'(\theta)=-\mean \bclc{(1-\theta)^{\bX(A)-1}\bX(A)}=-\int_A\mean\bclc{(1-\theta)^{\bX_x(A)-1}}\bmu(dx).\label{75}
}
Using~\eq{70}, we can construct~$\bX$ and~$\bX_x$ together such that~$\bX_x(A)\le \bX(A)+1$ a.s. Hence, it follows from~\eq{75}
that
\be{
 g'(\theta)\le -\int_A\mean\bclc{(1-\theta)^{\bX(A)}}\bmu(dx)=-g(\theta)\bmu(A).
}
However, $g(0)=1$, we obtain~$g(\theta)\le e^{-\theta\bmu(A)}$, which implies
\be{
 \prob[\bX(A)=0]=g(1)\le e^{-\bmu(A)}=e^{-|A|/\pi},
}
as claimed in~\eq{73}. In terms of~\eq{74}, we can use the construction~$\bX_x(A)\le \bX(A)+1$ a.s. again and the inequality~$(a+b)^{l-1}\le 2^{l-2}(a^{l-1}+b^{l-1})$ for all~$a,b\ge 0$ to obtain 
\bes{
\mean[\bX(A)^l]&=\mean\int_A\bX(A)^{l-1}\bX(dx)=\int_A\mean[\bX_x(A)^{l-1}]\bmu(dx)\\
&\le \int_A\mean [(\bX(A)+1)^{l-1}]\bmu(dx)=\bmu(A)\mean[(\bX(A)+1)^{l-1}]\\
&\le2^{l-2}\bmu(A)\{\mean[\bX(A)^{l-1}]+1\}\le 2^{l-1}\bmu(A)[1\vee\mean(\bX(A)^{l-1})].
}
Hence, \eq{74} follows by induction. 
\end{proof}

\begin{figure}
\begin{minipage}{0.33\textwidth}\centering
\begin{tikzpicture}[scale=2]
\begin{scope}[draw/.append style={thick}]
\draw (0,0) -- (  0:1);
\draw (0,0) -- ( 30:1);
\draw (0,0) -- ( 70:1);
\draw (0,0) -- (100:1);
\draw (0,0) -- (160:1);
\draw (0,0) -- (200:1);
\draw (0,0) -- (250:1);
\draw (0,0) -- (300:1);
\end{scope}
\filldraw [black] (0,0) circle (1pt);
\node at (125:0.25) {0};
\draw (30:0.4) arc (30:70:0.4);
\node at (50:0.6) {$\theta_i$};
\end{tikzpicture}
\end{minipage}\hfill
\begin{minipage}{0.33\textwidth}\centering
\begin{tikzpicture}[scale=2]
\draw[thick] (0,0) -- (0:1.2);
\draw[thick] (0,0) -- (60:1.2);
\draw (0,0) -- (10:0.9);
\filldraw (10:0.9) circle (1pt) node [right] {$v$};
\draw (0,0) -- (40:0.6);
\filldraw (40:0.6) circle (1pt) node [above right] {$y$};
\filldraw (0,0) circle (1pt);
\node at (125:0.2) {0};
\node at (30:1.1) {\large$A_i$};
\draw (10:0.35) arc (10:40:0.35);
\node at (25:0.47) {$\theta$};
\end{tikzpicture}
\end{minipage}\hfill
\begin{minipage}{0.33\textwidth}\centering
\begin{tikzpicture}[scale=2.2]
\filldraw (0,0) circle (1pt);
\draw[dashed,thick] (1,0) arc (0:360:1);
\foreach \x in {0,30,60,90,120,150,180,210,240,270,300,330}
	\draw[thick] (0,0) -- (\x:1);
\node at (10:0.7) {$A_{x,i}(t)$};
\draw[latex-latex] (165:0.05) -- (165:1);
\node[fill=white,inner sep=2pt] at (164:0.52) {$t$};
\node[fill=white,inner sep=2pt] at (245:0.25) {$x$};
\end{tikzpicture}
\end{minipage} 
\vskip1\baselineskip
\begin{minipage}{0.33\textwidth}
\caption{\label{fig3}Isosceles triangles.}
\end{minipage}\hfill
\begin{minipage}{0.33\textwidth}
\caption{\label{fig4}$\theta\le\pi/3$.\\\strut}
\end{minipage}\hfill
\begin{minipage}{0.33\textwidth}
\caption{\label{fig5}Twelve disjoint congruent equal sectors.} %
\end{minipage}\hfill
\end{figure}

\begin{lma}\label{lem2} Suppose that rays emanating from the center~$0$ of~$\C$ divide~$\C$ 
into disjoint congruent isosceles triangles~$A_i$ with angles~$\theta_i$, $i=1,\dots,k$ (see Figure~\ref{fig3}), where~$k$ may be finite or infinity. If~$\bx\subset\C$ satisfies~$\bx\cap A_i\ne\emptyset$ and~$\theta_i\le \pi/3$ for all~$i=1,\dots,k$, then the Voronoi cell~$\cc(0,\bx)$ is contained in the disk~$B(0,d(\bx))$, where~$B(x,r)\coloneqq \{u\in\C:\ |u-x|\le r\}$ and~$d(\bx)\coloneqq \sup\{|u|:\ u\in\bx\}$.
\end{lma}

\begin{proof} For each~$y\in \cc(0,\bx)$, there exists a triangle~$A_i$ such that~$y\in A_i$. Since
$\bx\cap A_i\ne\emptyset$, there exists a point~$v\in\bx\cap A_i$ (see Figure~\ref{fig4}) and it follows from~$y\in \cc(0,\bx)$ that~$|y|\le|y-v|$. This in turn implies
$|v|\ge2|y|\cos(\theta)\ge |y|$, that is, $y\in B(0,d(\bx))\coloneqq \{u:\ |u|\le d(\bx)\}$. This completes the proof.
\end{proof}

\begin{proof}[Proof of Theorem~\ref{thm4}] From the definition of the Ginibre point process, it is a Gibbs point process with a pair potential function (\cite{OS15}), hence~\eq{71}  is direct 
corollaries of Theorem~2.1 of \cite{SY13} and Theorem~1.1 of \cite{XY15}. Hence, it remains to show~\eq{72}.

We divide the disk~$B^o(x,t)\coloneqq \{u:\ 0<|u-x|<t\}$ into twelve disjoint congruent equal sectors~$A_{x,i}(t)$, $i=1,\dots,12$, (see Figure~\ref{fig5}) and define
\be{
 T_{x}\coloneqq \inf\{t:\ A_{x,i}(t)\cap \bX\ne \emptyset,\ i=1,\dots,12\},
}
(see \cite{MY99}  and \cite{PY01}).
The area of~$A_{x,i}(t)$ is~$|A_{x,i}(t)|=\pi t^2/{12}$, so it follows from Lemma~\ref{lem1} that 
\besn{\label{76}
\prob[T_{x}>t]&=\prob\left(\cup_{i=1}^{12}\{\bX(A_{x,i}(t))=0\}\right)\\
&\le 12\,\prob[\bX(A_{x,1}(t))=0]
\le 12e^{-t^2/{12}}.
}
Write~$Y_x\coloneqq |\Xi_x|-|\Xi|$, we establish that, for~$k\in\N$, 
\ben{
\mean |Y_x|^k\le C(k),\label{77}
}
where~$C(k)$ is a constant dependent on~$k$. 
In fact, for a bounded measurable function~$f$ on the space of all locally finite measures on~$\C$, a routine exercise of random measures gives
\bes{
	\mean f(\Xi_x)
		&=\frac{\mean\left\{ f\left(\sum_{y\in \bX\cap Q_\lambda}L(y,\bX)\right)L(x,\bX)\bX(dx)\right\}}{\mean\left\{ L(x,\bX)\bX(dx)\right\}}\\
		&=\frac{\mean\left\{ f\left(\sum_{y\in \bX_x\cap Q_\lambda}L(y,\bX_x)\right)L(x,\bX_x)\right\}}{\mean L(x,\bX_x)}.
}
That is, if the edges of the Voronoi tessellations are not affected by moving the point from~$x+Z$ to~$x$, then their distribution is not affected either. 
On the other hand,  adding a point at~$x$ does not affect the Voronoi cells centered at points outside~$B(x,3T_x)$ 
and deleting a point at~$x+Z$ does not affect the Voronoi cells centered at points outside~$B(x+Z,3T_{x+Z})$  (see \cite[Section 4]{MY99}). Therefore,
by Lemma~\ref{lem2}, 
the change of Voronoi edge lengths due to shifting a point at~$x+Z$ to~$x$ is bounded by~$2\pi [T_x\bX(B(x,3T_x))+T_{Z+x}\bX(B(x+Z,3T_{x+Z}))]$, so 
\besn{\label{78}
\mean |Y_x|^k&\le(2\pi)^k\mean\bclc{\bclr{T_x\bX(B(x,3T_x))+T_{Z+x}\bX(B(x+Z,3T_{x+Z}))}^k}\\
&\le 0.5(4\pi)^k\mean \left\{T_x^k\bX(B(x,3T_x))^k+T_{Z+x}^k\bX(B(x+Z,3T_{x+Z}))^k\right\}\\
&= (4\pi)^k\mean \left\{T_0^k\bX(B(0,3T_0))^k\right\},
}
where the second inequality follows from the fact that~$(a+b)^k\le 2^{k-1}(a^k+b^k)$ for all~$a,b\ge 0$ and the equality 
holds because~$T_{Z+x}\bX(B(x+Z,3T_{x+Z}))$ and~$T_x\bX(B(x,3T_x))$ have the same distribution as that of~$T_0\bX(B(0,3T_0))$.
However,
\besn{\label{79}
&\mean \left\{T_0^k\bX(B(0,3T_0))^k\right\}\\
&\qquad\le12\mean\int_{\C}|y|^k\bX(B(0,3|y|))^k\\
&\kern8em\times\bone\left[\{\bX(A_{0,1}(|y|))=0\}\cap\bigcap_{i=2}^{12}\{\bX(A_{0,i}(|y|))\ge 1\}\right]\bX(dy)\\
&\qquad\le12\mean\int_{\C}|y|^k\bX(B(0,3|y|))^k\bone[\bX(A_{0,1}(|y|))=0]\bX(dy).
}
We apply the Georgii-Nguyen-Zessin
integral characterization of Gibbs point processes (see \cite[Chapter 6.4]{MW}) to obtain that the
conditional probability of observing an extra point of~$\bX$ in the volume element~$dy$, given that configuration without
that point, equals
~$\pi^{-1}\exp( - \beta \Delta^{{\Psi}}( \{y\}, \bX
) )dy\le \pi^{-1}dy$, where~$\Delta^{{\Psi}}( \{y\}, \bX
)\ge 0$ is the local energy function and~$1/\beta\ge 0$ is the temperature (see \cite[Section~1.1]{XY15}). Hence, it follows from
\eq{79} that
\bes{
&\mean \left\{T_0^k\bX(B(0,3T_0))^k\right\}\\
&\qquad\le\frac{12}\pi\mean\int_{\C}|y|^k(\bX+\delta_y)(B(0,3|y|))^k\\
&\kern11em\times\bone[(\bX+\delta_y)(A_{0,1}(|y|))=0]\exp(- \beta \Delta^{{\Psi}}( \{y\}, \bX
) )dy\\
&\qquad\le\frac{12}\pi\mean\int_{\C}|y|^k(\bX+\delta_y)(B(0,3|y|))^k\bone[(\bX+\delta_y)(A_{0,1}(|y|))=0]dy\\
&\qquad=24\mean\int_0^\infty t^k(\bX+\delta_y)(B(0,3t))^k\bone[\bX(A_{0,1}(t))=0]dt\\
&\qquad\le24\cdot 2^{k-1}\mean\int_0^\infty t^k\{\bX(B(0,3t))^k+1\}\bone[\bX(A_{0,1}(t))=0]dt\\
&\qquad\le 24\cdot 2^{k-1}\int_0^\infty t^k\sqrt{\mean\clc{\bX(B(0,3t))^{2k}}\prob[\bX(A_{0,1}(t))=0]}\,dt\\
&\qquad \ \ \ +24\cdot 2^{k-1}\int_0^\infty t^k\prob[\bX(A_{0,1}(t))=0]\,dt\\
&\qquad\le 12\times2^{0.5k(2k+1)}\int_0^\infty \max\bclc{1,9^kt^{2k}}t^ke^{-t^2/24}dt+24\cdot 2^{k-1}\int_0^\infty t^k e^{-t^2/12}dt,
}
where the last inequality follows from~\eq{73}. This, together with 
\eq{78}, yields the bound in \eq{77}.

We now apply Theorem~\ref{thm2} to establish~\eq{72}.
\medskip

\noindent{\it The estimate of~$r_1'$.} To simplify the notation, we write~$Y_x'=Y_x\bone[|Y_x|\le B]$. 
Set 
\be{
 U_x=\{T_x\le 4\sqrt{3\ln\lambda}\}\cap \{T_{Z+x}\le 4\sqrt{3\ln\lambda}\}\cap \{|Z|\le 2\sqrt{3\ln \lambda}\}.
}
Using~\eq{76}, we have
\besn{\label{80}
	\prob[U_x^c]&=\prob[U_y^c]\\
	&\le\prob[T_x>4\sqrt{3\ln\lambda}]+\prob[T_{x+Z}>4\sqrt{3\ln\lambda}]+\prob[|Z|>2\sqrt{3\ln\lambda}]
\\
&=\bigo(\lambda^{-4}).
}
This, together with~\eq{77}, ensures
\besn{
&\cov\left(Y_x'\bone[U_x^c],Y_y'\right)\\
&\qquad\le\sqrt{\var\left(Y_x'\bone[U_x^c]\right)\var\left(Y_y'\right)}\le \sqrt{\mean\clc{(Y_x')^2\bone[U_x^c]}\mean\clc{(Y_y')^2}}\\
&\qquad\le \sqrt{\mean\clc{Y_x^2\bone[U_x^c]}\mean Y_y^2}\le\sqrt{\sqrt{\mean Y_x^4 \prob[U_x^c]}\mean Y_y^2 }
	=\bigo(\lambda^{-1}).\label{81}
}
Similarly, we can derive
\ben{
	\cov\left(Y_x'\bone[U_x],Y_y'\bone[U_y^c]\right)
	=\bigo(\lambda^{-1})\label{82}
}
and 
\besn{
\cov\left(Y_x'\bone[U_x],Y_y'\bone[U_y]\right)
	& \le\sqrt{\var\left(Y_x'\bone[U_x]\right)\var\left(Y_y'\bone[U_y]\right)} \\
 &\le \sqrt{\mean Y_x^2 \mean Y_y^2 }
=\bigo(1).\label{83}
}
Assume~$|x-y|> 20\sqrt{3\ln\lambda}$. Conditional on~$U_xU_y$, 
$Y_x$ is independent of~$Y_y$, hence
\be{
 \mean(Y_x'Y_y'|U_xU_y)=\mean(Y_x'|U_xU_y)\mean(Y_y'|U_xU_y).
}
Using~\eq{77} and~\eq{80}, we obtain
\be{
 |\mean(Y_x'\bone[U_xU_y^c])|\le\{\mean [(Y_x'\bone[U_x])^2]\prob[U_y^c]\}^{1/2}\le \{\mean [Y_x^2]\prob[U_y^c]\}^{1/2}=\bigo(\lambda^{-2}).
}
The same argument gives that
all~$\mean(Y_x'\bone[U_xU_y])$, $\mean(Y_y'\bone[U_xU_y])$ and~$\mean(Y_x'\bone[U_x])$ are of order~$\bigo(1)$ and~$\mean(Y_y'\bone[U_x^cU_y])=\bigo(\lambda^{-2})$, hence
\besn{\label{84}
&\cov(Y_x'\bone[U_x],Y_y'\bone[U_y])\\
&\qquad=
\mean(Y_x'|U_xU_y)\mean(Y_y'|U_xU_y)\prob[U_xU_y]-\mean(Y_x'\bone[U_xU_y])\mean(Y_y'\bone[U_xU_y])\\
&\qquad\quad-\mean(Y_x'\bone[U_xU_y^c])\mean(Y_y'\bone[U_xU_y])-\mean(Y_x'\bone[U_x])\mean(Y_y'\bone[U_x^cU_y])\\
&\qquad=
\mean(Y_x'\bone[U_xU_y])\mean(Y_y'\bone[U_xU_y])[\prob[U_xU_y]]^{-1}(1-\prob[U_xU_y])\\
&\qquad\quad+\bigo(\lambda^{-2})+\bigo(\lambda^{-2})\\
&\qquad=\bigo\left(\lambda^{-2}\right),
}
where the last equation follows from~\eq{80} since~$1-\prob[U_xU_y]\le \prob[U_x^c]+\prob[U_y^c]=o(\lambda^{-4})$.
Now,
\bea{
&\iint_{\Gamma^2} \cov(Y_x',Y_y')\Lambda(dx)\Lambda(dy)\\
&\qquad=\iint_{\Gamma^2} \cov(Y_x'\bone[U_x^c],Y_y')\Lambda(dx)\Lambda(dy)\\
&\qquad\quad+\iint_{\Gamma^2} \cov(Y_x'\bone[U_x],Y_y'\bone[U_y^c])\Lambda(dx)\Lambda(dy)\\
&\qquad\quad+\iint_{|x-y|>20\sqrt{3\ln\lambda}} \cov(Y_x'\bone[U_x],Y_y'\bone[U_y])\Lambda(dx)\Lambda(dy)\\
&\qquad\quad+\iint_{|x-y|\le 20\sqrt{3\ln\lambda}} \cov(Y_x'\bone[U_x],Y_y'\bone[U_y])\Lambda(dx)\Lambda(dy).
}
Using~\eq{81} for the first term, \eq{82} for the second term, \eq{84} for the third term and~\eq{83} for the last term, we have
\be{
 \iint_{\Gamma^2} \cov(Y_x',Y_y')\Lambda(dx)\Lambda(dy)=\bigo(\lambda\ln\lambda).
}
This gives the estimate of
$r_1'$ as
\ben{
r_1'=\bigo(\lambda^{-1})\left[\iint_{\Gamma^2} \cov(Y_x',Y_y')\Lambda(dx)\Lambda(dy)\right]^{1/2}
=\bigo\left(\lambda^{-1/2}\sqrt{\ln\lambda}\right).\label{85}
}

\medskip
\noindent\emph{The estimate of~$r_2'+r_3'$.} Applying~\eq{77} gives
\ben{
r_2'+r_3'\le B^{-3}\int_\Gamma \mean Y_x^2\Lambda(dx)=\bigo\left(\lambda^{-1/2}\right).
\label{86}
}

\medskip
\noindent\emph{The estimate of~$r_4'$.} To simplify the notation, we write 
\be{
\zeta_{x,t}=\begin{cases}
\bone[1\ge\Delta_x>t>0] &\text{for~$t>0$,}\\
\bone[-1\le\Delta_x<t<0] &\text{for~$t<0$.}
\end{cases}
}
If~$|x-y|>20\sqrt{3\ln\lambda}$, we have
\besn{
	\cov\left(\zeta_{x,t},\zeta_{y,t}\right)&=\cov\left(\zeta_{x,t}\bone[U_x^c],\zeta_{y,t}\right)
+\cov\left(\zeta_{x,t}\bone[U_x],\zeta_{y,t}\bone[U_y^c]\right)\\
	&\quad+\cov\left(\zeta_{x,t}\bone[U_x],\zeta_{y,t}\bone[U_y]\right).\label{87}
}
We apply~\eq{80} to obtain
\ben{\cov\left(\zeta_{x,t}\bone[U_x^c],\zeta_{y,t}\right)
\le\prob[U_x^c]=\bigo\left(\lambda^{-4}\right).\label{88}
}
Likewise, 
\be{\cov\left(\zeta_{x,t}\bone[U_x],\zeta_{y,t}\bone[U_y^c]\right)=\bigo\left(\lambda^{-4}\right).
}
Given~$U_xU_y$, $\zeta_{x,t}$ is independent of~$\zeta_{y,t}$, hence 
\be{
 \mean(\zeta_{x,t}\zeta_{y,t}|U_xU_y)=\mean(\zeta_{x,t}|U_xU_y)\mean(\zeta_{y,t}|U_xU_y).
}
This ensures that we can expand the last term of~\eq{87} into
\besn{
&\cov\left(\zeta_{x,t}\bone[U_x],\zeta_{y,t}\bone[U_y]\right)\\
&\qquad=\mean(\zeta_{x,t}|U_xU_y)\mean(\zeta_{y,t}|U_xU_y)\prob[U_xU_y]\left(1-\prob[U_xU_y]\right)\\
&\qquad\quad-\mean(\zeta_{x,t}|U_x)\prob[U_x]\mean (\zeta_{y,t}|U_x^cU_y)\prob[U_x^cU_y]\\
&\qquad\quad-\mean(\zeta_{x,t}|U_xU_y^c)\prob[U_xU_y^c]\mean (\zeta_{y,t}|U_xU_y)\prob[U_xU_y]\\
&\qquad=\bigo\left(\prob[U_x^c]+\prob[U_y^c]\right)
	=\bigo\left(\lambda^{-4}\right),\label{89}
}
again, by~\eq{80}. Combining estimates~\eq{88}--\eq{89}, we obtain from~\eq{87} that, when~$|x-y|>20\sqrt{3\ln\lambda}$,
\ben{\cov\left(\zeta_{x,t},\zeta_{y,t}\right)=\bigo\left(\lambda^{-4}\right).\label{90}
}
This implies
\besn{
&\int_0^1\iint_{\Gamma^2}\cov\left(\zeta_{x,t},\zeta_{y,t}\right)\Lambda(dx)\Lambda(dy)dt\\
&\qquad=\int_0^1\iint_{|x-y|>20\sqrt{3\ln\lambda}}\cov\left(\zeta_{x,t},\zeta_{y,t}\right)\Lambda(dx)\Lambda(dy)dt\\
&\qquad\quad+\int_0^1\iint_{|x-y|\le20\sqrt{3\ln\lambda}}\cov\left(\zeta_{x,t},\zeta_{y,t}\right)\Lambda(dx)\Lambda(dy)dt\\
&\qquad\le \bigo\left(\lambda^{-2}\right)+\int_0^1\iint_{|x-y|\le20\sqrt{3\ln\lambda}}\mean\zeta_{x,t}\Lambda(dx)\Lambda(dy)dt
\\
&\qquad=\bigo\left(\lambda^{-2}\right)+B^{-1}\iint_{|x-y|\le20\sqrt{3\ln\lambda}}\mean|Y_x|\Lambda(dx)\Lambda(dy)
\\
&\qquad\le \bigo\left(\lambda^{1/2}\ln\lambda\right),\label{91}
}
where the last inequality is due to~\eq{77}.
Similarly, we can also establish
\ben{
\int_{-1}^0\iint_{\Gamma^2}\cov\left(\zeta_{x,t},\zeta_{y,t}\right)\Lambda(dx)\Lambda(dy)dt
=\bigo\left(\lambda^{1/2}\ln\lambda\right).\label{92}
}
Adding~\eq{91} and~\eq{92} gives
\ben{
r_4'=\bigo\left(\lambda^{-1/2}\ln\lambda\right).\label{93}
}

\medskip
\noindent\emph{The estimate of~$r_5'$.} We make use of~\eq{90} to get
\besn{
&\int_0^1\iint_{\Gamma^2}t\cov\left(\zeta_{x,t},\zeta_{y,t}\right)\Lambda(dx)\Lambda(dy)dt\\
&\qquad=\int_0^1\iint_{|x-y|>20\sqrt{3\ln\lambda}}t\cov\left(\zeta_{x,t},\zeta_{y,t}\right)\Lambda(dx)\Lambda(dy)dt\\
&\qquad\quad+\int_0^1\iint_{|x-y|\le20\sqrt{3\ln\lambda}}t\cov\left(\zeta_{x,t},\zeta_{y,t}\right)\Lambda(dx)\Lambda(dy)dt\\
&\qquad\le \bigo\left(\lambda^{-2}\right)+\int_0^1\iint_{|x-y|\le20\sqrt{3\ln\lambda}}t\mean\zeta_{x,t}\Lambda(dx)\Lambda(dy)dt
\\
&\qquad=\bigo\left(\lambda^{-2}\right)+\frac{B^{-2}}2\iint_{|x-y|\le20\sqrt{3\ln\lambda}}\mean|Y_x|^2\Lambda(dx)\Lambda(dy)
\\
&\qquad\le \bigo\left(\ln\lambda\right),\label{94}
}
where, again, the last inequality follows from~\eq{77}. Correspondingly, we can deduce the following bound:
\ben{
\int_{-1}^0\iint_{\Gamma^2}t\cov\left(\zeta_{x,t},\zeta_{y,t}\right)\Lambda(dx)\Lambda(dy)dt=\bigo\left(\ln\lambda\right).\label{95}
}
Combining~\eq{94} and~\eq{95} yields
\ben{
r_5'=\bigo\left(\lambda^{-1/2}\sqrt{\ln\lambda}\right).\label{96}
}

Finally, we collect all the estimates in~\eq{85}, \eq{86},
~\eq{93} and~\eq{96} to achieve the bound~\eq{72}, as claimed. 
\end{proof}

\subsection{The total edge length of Poisson-Voronoi tessellations}

The Poisson-Voronoi tessellations have been studied extensively since \cite{AB93}. For normal approximation of the total edge length of Poisson-Voronoi tessellations, an error bound of the optimal order was established by \cite{LSY17} using the Malliavin-Stein approach. In this subsection, we demonstrate that Theorem~\ref{thm2} can be utilized to derive an error bound of the same order.

Similar to the previous subsection, we consider $\bX$ as a Poisson point process on $\R^2$ with mean measure~$\bmu(dx)=dx$ and set $Q_\lambda=\{(x_1,x_2):\ -\sqrt{\lambda}/2\le x,y\le \sqrt{\lambda}/2\}$. For the ease of reading, we briefly recap a few essential terminologies. For a realization~$\bx$ of~$\bX$ and~$x\in\bx$, we define 
$\cc(x,\bx)$ as the set of every point in~$\R^2$ whose (Euclidean) distance to~$x$ is less than or equal to its distance to any other point of~$\bx$. 
The collection of~$\cc(x,\bx)$, $x\in\bx$, is called the \emph{Poisson-Voronoi tessellation} induced by the realization~$\bx$ of~$\bX$. 
Again, we write $L(x,\bX)\coloneqq L_\lambda(x,\bX)$ as one half the total edge length of the \emph{finite} length edges in the cell~$\cc(x,\bX)$, then the total edge length of the Poisson-Voronoi tessellation induced by~$\bX$ with centers in~$\bX\cap Q_\lambda$ can be summarized as 
\be{
 {{\cl(\lambda)}}\coloneqq \int_{Q_\lambda}L(x,\bX)\bX(dx).
}

\begin{thm} Let~$B^2=\var({{\cl(\lambda)}})$ and~$W=({{\cl(\lambda)}}-\mean {{\cl(\lambda)}})/B$. We have
\ben{\lim_{\lambda\to\infty}\lambda^{-1} \mean {{\cl(\lambda)}}\in(0,\infty), \ \ \ 
\lim_{\lambda\to\infty}\lambda^{-1}B^2\in(0,\infty)\label{add1}
} 
and
\ben{
\dk({\law}(W),\cn(0,1))=\bigo\left(\lambda^{-1/2}\right).
\label{add2}
}
\end{thm}

\begin{proof} The claim \eq{add1} can be found in \cite{AB93}, hence it suffices to show \eq{add2}. To this end, we write $\Xi(dx)=L(x,\bX)\bX(dx)$ {and apply Theorem~\ref{thm2}.}
As observed by \cite[Section 4]{MY99}, adding a point at~$\alpha$ does not affect the Voronoi cells centered at points outside~$B(\alpha,3T_\alpha)$, while $Y_\alpha$ (respectively $Y_\beta$) is determined by the configuration $B(\alpha, 3T_\alpha)\cap \bX$ (respectively $B(\beta, 3T_\beta)\cap \bX$). Thus $Y_\alpha$ and $Y_\beta$ are conditionally independent given~$T_\alpha$ and $T_\beta$ with $|\alpha-\beta|>3(T_\alpha+T_\beta)$, which implies
\besn{
&\cov(Y_\alpha',Y_\beta')\\
&=\int_0^\infty \int_0^\infty\mean (Y_\alpha' Y_\beta'|T_\alpha=s_1,T_\beta=s_2)\prob(T_\alpha\in ds_1,T_\beta\in ds_2)\\
&\qquad-\int_0^\infty \mean (Y_\alpha'|T_\alpha=s_1)\prob(T_\alpha\in ds_1)\int_0^\infty \mean (Y_\beta' |T_\beta=s_2)\prob(T_\beta\in ds_2)\\
&=\iint_{3(s_1+s_2)\ge |\alpha-\beta|} \mean (Y_\alpha' Y_\beta' |T_\alpha=s_1,T_\beta=s_2)\prob(T_\alpha\in ds_1,T_\beta\in ds_2)\\
&\qquad -\iint_{3(s_1+s_2)\ge |\alpha-\beta|} \mean (Y_\alpha' |T_\alpha=s_1)\mean (Y_\beta' |T_\beta=s_2)\prob(T_\alpha\in ds_1)\prob(T_\beta\in ds_2).\label{PV01}}
However, direct verification ensures
\ben{\prob(T_\alpha\in ds)=2\pi\left(1-e^{-\frac{\pi s^2}{12}}\right)^{11}e^{-\frac{\pi s^2}{12}}sds.\label{PV02}}
By checking the relationship of the event $\{T_\beta\in ds_2\}$ and various possible cases of the event $\{T_\alpha=s_1\}$, we obtain
\be{
\prob(T_\beta\in ds_2|T_\alpha=s_1)\le \frac{e^{-\frac{\pi s_2^2}{12}}\frac{2\pi s_2}{12}ds_2}{1-e^{-\frac{\pi s_1^2}{12}}},}
which, together with \eq{PV02}, implies
\besn{&\prob(T_\alpha\in ds_1,T_\beta\in ds_2)=\prob(T_\beta\in ds_2|T_\alpha=s_1)\prob(T_\alpha\in ds_1)\\
&\qquad\le \frac{\pi^2}{3} s_1s_2e^{-\frac{\pi(s_1^2+s_2^2)}{12}}ds_1ds_2. \label{PV03}
}
Since the change of the total edge lengths as the result of adding a point at $\alpha$ (respectively $\beta$) can be bounded by $2\pi T_\alpha$ (respectively $2\pi T_\beta$), if we use $C$ to represent a constant independent of $\lambda$, $\alpha$ and $\beta$, whose value may vary from one line to another, then we have the following crude estimates for $Y'$ and the same estimates also hold for $Y$ in place of $Y'$:
\besn{
&|\mean(Y_\alpha' |T_\alpha=s_1,T_\beta=s_2)|\le  
C(s_1^3+1),\\
&|\mean(Y_\alpha' |T_\alpha=s_1)|\le  C(s_1^3+1),\\
&|\mean(Y_\alpha' Y_\beta' |T_\alpha=s_1,T_\beta=s_2)|\le C(s_1^3s_2^3+1),\\
&|\mean(Y_\alpha'^2|T_\alpha=s_1,T_\beta=s_2)|\le C(s_1^4+1),\\
&|\mean(Y_\alpha'^2|T_\alpha=s_1)|\le C(s_1^4+1).\label{PV03.1}
} 
Combining \eq{PV01}, \eq{PV02}, \eq{PV03} and \eq{PV03.1} gives
\besn{
|\cov(Y_\alpha' ,Y_\beta' )|&\le C\iint_{3(s_1+s_2)\ge |\alpha-\beta|}(s_1^4s_2^4+1) e^{-\frac{\pi(s_1^2+s_2^2)}{12}}ds_1ds_2\\
&\le C(|\alpha-\beta|^3+1)e^{-\frac{\pi|\alpha-\beta|^2}{432}}.\label{PV04}}
Using \eq{PV04}, we have
\besn{r_1'&\le O\left(\lambda^{-1}\right)\left(\iint_{Q_\lambda\times Q_\lambda}\cov(Y_\alpha',Y_\beta')\Lambda(d\alpha) \Lambda(d\beta)\right)^{1/2}\\
&\le O\left(\lambda^{-1}\right)\left(\iint_{Q_\lambda\times Q_\lambda}(|\alpha-\beta|^3+1)e^{-\frac{\pi|\alpha-\beta|^2}{432}}\Lambda(d\alpha) \Lambda(d\beta)\right)^{1/2}=O(\lambda^{-1/2}).\label{PV06}}
For $r_2'$ and $r_3'$, we use \eq{PV02} and \eq{PV03.1} to obtain
\bea{&\mean (Y_\alpha^2)=\int_0^\infty \mean(Y_\alpha^2|T_\alpha=s)\prob(T_\alpha\in ds)\le 
\int_0^\infty C(s^4+1)e^{-\frac{\pi s^2}{12}}sds=C,\\
&\mean (Y_\alpha'^2)=\int_0^\infty \mean(Y_\alpha'^2|T_\alpha=s)\prob(T_\alpha\in ds)\le 
\int_0^\infty C(s^4+1)e^{-\frac{\pi s^2}{12}}sds=C,}
which implies
$$r_2'\le O\left(\lambda^{-3/2}\right)\int_{Q_\lambda}\mean (Y_\alpha'^2) \Lambda(d\alpha)= O\left(\lambda^{-1/2}\right)$$
and
$$r_3'\le O\left(\lambda^{-3/2}\right)\int_{Q_\lambda}\mean (Y_\alpha^2) \Lambda(d\alpha)=O\left(\lambda^{-1/2}\right).$$
For $r_4'$, when $3(s_1+s_2)<|\alpha-\beta|$, we have
\bea{&\prob(1\ge \Delta_\alpha>t>0,1\ge \Delta_\beta>t>0|T_\alpha=s_1,T_\beta=s_2)\\
&=\prob(1\ge \Delta_\alpha>t>0|T_\alpha=s_1)\prob(1\ge \Delta_\beta>t>0|T_\beta=s_2).}
Hence
\bea{
&\cov(\bone[1\ge \Delta_\alpha>t>0],\bone[1\ge \Delta_\beta>t>0])\\
&= \int_0^\infty\int_0^\infty\prob(1\ge \Delta_\alpha>t>0,1\ge \Delta_\beta>t>0|T_\alpha=s_1,T_\beta=s_2)\\
&\qquad\qquad\qquad\qquad\qquad\qquad\qquad\qquad\qquad\qquad\prob(T_\alpha\in ds_1,T_\beta\in ds_2)\\
&\ \ \ -\left(\int_0^\infty\prob(1\ge \Delta_\alpha>t>0|T_\alpha=s)\prob(T_\alpha\in ds)\right)^2\\
&= \iint_{3(s_1+s_2)\ge|\alpha-\beta|}\prob(1\ge \Delta_\alpha>t>0,1\ge \Delta_\beta>t>0|T_\alpha=s_1,T_\beta=s_2)\\
&\qquad\qquad\qquad\qquad\qquad\qquad\qquad\qquad\qquad\qquad\prob(T_\alpha\in ds_1,T_\beta\in ds_2)\\
&\ \ \ -\iint_{3(s_1+s_2)\ge|\alpha-\beta|}\prob(1\ge \Delta_\alpha>t>0|T_\alpha=s_1)\prob(1\ge \Delta_\beta>t>0|T_\beta=s_2)\\
&\qquad\qquad\qquad\qquad\qquad\qquad\qquad\qquad\qquad\qquad\prob(T_\alpha\in ds_1)\prob(T_\beta\in ds_2),}
which implies
\besn{
&\int_0^1|\cov(\bone[1\ge \Delta_\alpha>t>0],\bone[1\ge \Delta_\beta>t>0])|dt\\
&\le \iint_{3(s_1+s_2)\ge|\alpha-\beta|}\int_0^1\prob(1\ge \Delta_\alpha>t>0,1\ge \Delta_\beta>t>0|T_\alpha=s_1,T_\beta=s_2)dt\\
&\qquad\qquad\qquad\qquad\qquad\qquad\qquad\qquad\qquad\qquad\prob(T_\alpha\in ds_1,T_\beta\in ds_2)\\
&\ \ \ +\iint_{3(s_1+s_2)\ge|\alpha-\beta|}\int_0^1\prob(1\ge \Delta_\alpha>t>0|T_\alpha=s_1)\prob(1\ge \Delta_\beta>t>0|T_\beta=s_2)dt\\
&\qquad\qquad\qquad\qquad\qquad\qquad\qquad\qquad\qquad\qquad\prob(T_\alpha\in ds_1)\prob(T_\beta\in ds_2)\\
&\le 2\iint_{3(s_1+s_2)\ge|\alpha-\beta|,s_1<s_2}\mean(|\Delta_\alpha||T_\alpha=s_1,T_\beta=s_2)\prob(T_\alpha\in ds_1,T_\beta\in ds_2)\\
&\ \ \ +2\iint_{3(s_1+s_2)\ge|\alpha-\beta|,s_1<s_2}\mean(|\Delta_\alpha||T_\alpha=s_1)\prob(T_\alpha\in ds_1)\prob(T_\beta\in ds_2)\\
&\le O\left(\lambda^{-1/2}\right)\iint_{3(s_1+s_2)\ge|\alpha-\beta|,s_1<s_2}(s_1^3+1)\prob(T_\alpha\in ds_1,T_\beta\in ds_2)\\
&\ \ \ +O\left(\lambda^{-1/2}\right)\iint_{3(s_1+s_2)\ge|\alpha-\beta|,s_1<s_2}(s_1^3+1)\prob(T_\alpha\in ds_1)\prob(T_\beta\in ds_2)\\
&\le O\left(\lambda^{-1/2}\right)e^{-\frac{\pi|\alpha-\beta|^2}{432}},
\label{PV05}}
where the second last inequality is from \eq{PV03.1} and the last inequality is obtained as in \eq{PV04}. Thus, it follows from \eq{PV05} and the same argument for \eq{PV06} that
\besn{
&\left|\iint_{Q_\lambda\times Q_\lambda}\int_0^1\cov(\bone[1\ge \Delta_\alpha>t>0],\bone[1\ge \Delta_\beta>t>0])dt \Lambda(d\alpha)  \Lambda(d\beta)\right|\\
&\le O\left(\lambda^{-1/2}\right)\iint_{Q_\lambda\times Q_\lambda}e^{-\frac{\pi|\alpha-\beta|^2}{432}} \Lambda(d\alpha)  \Lambda(d\beta)\\
&=O\left(\lambda^{1/2}\right).\label{PV07}
}
Likewise, we can show that
\besn{
&\left|\iint_{Q_\lambda\times Q_\lambda}\int_{-1}^0\cov(\bone[-1\le \Delta_\alpha<t<0],\bone[-1\le \Delta_\beta<t<0])dt \Lambda(d\alpha)  \Lambda(d\beta)\right|\\
&=O\left(\lambda^{1/2}\right).\label{PV08}
}
Combining \eq{PV07} and \eq{PV08} gives $r_4'{=} O\left(\lambda^{-1/2}\right)$.
For $r_5'$, we can bound it in the same way as for $r_4'$. In fact, we replace \eq{PV05} with
\bea{
&\int_0^1t|\cov(\bone[1\ge \Delta_\alpha>t>0],\bone[1\ge \Delta_\beta>t>0])|dt\\
&\le \iint_{3(s_1+s_2)\ge |\alpha-\beta|}\int_0^1t\prob(1\ge \Delta_\alpha>t>0,1\ge \Delta_\beta>t>0|T_\alpha=s_1,T_\beta=s_2)dt\\
&\qquad\qquad\qquad\qquad\qquad\qquad\qquad\qquad\qquad\qquad\prob(T_\alpha\in ds_1,T_\beta\in ds_2)\\
&\ \ \ +\iint_{3(s_1+s_2)\ge |\alpha-\beta|}\int_0^1t\prob(1\ge \Delta_\alpha>t>0|T_\alpha=s_1)\prob(1\ge \Delta_\beta>t>0|T_\beta=s_2)dt\\
&\qquad\qquad\qquad\qquad\qquad\qquad\qquad\qquad\qquad\qquad\prob(T_\alpha\in ds_1)\prob(T_\beta\in ds_2)\\
&\le \iint_{3(s_1+s_2)\ge |\alpha-\beta|,s_1<s_2}\mean(\Delta_\alpha^2|T_\alpha=s_1,T_\beta=s_2)\prob(T_\alpha\in ds_1,T_\beta\in ds_2)\\
&\ \ \ +\iint_{3(s_1+s_2)\ge |\alpha-\beta|,s_1<s_2}\mean(\Delta_\alpha^2|T_\alpha=s_1)\prob(T_\alpha\in ds_1)\prob(T_\beta\in ds_2)\\
&\le O\left(\lambda^{-1}\right) \iint_{3(s_1+s_2)\ge |\alpha-\beta|,s_1<s_2}(s_1^4+1)\prob(T_\alpha\in ds_1,T_\beta\in ds_2)\\
&\ \ \ +O\left(\lambda^{-1}\right)\iint_{3(s_1+s_2)\ge |\alpha-\beta|,s_1<s_2}(s_1^4+1)\prob(T_\alpha\in ds_1)\prob(T_\beta\in ds_2)\\
&{=} O\left(\lambda^{-1}\right)e^{-\frac{\pi|\alpha-\beta|^2}{432}},
}
where the second last inequality is from \eq{PV03.1} and the last inequality is from a similar argument leading to \eq{PV04}. Therefore,
\besn{
&\left|\int_0^1\iint_{Q_\lambda\times Q_\lambda}t\cov(\bone[1\ge \Delta_\alpha>t>0],\bone[1\ge \Delta_\beta>t>0])dt \Lambda(d\alpha)  \Lambda(d\beta)\right|\\
&\le O\left(\lambda^{-1}\right)\iint_{Q_\lambda\times Q_\lambda}e^{-\frac{\pi|\alpha-\beta|^2}{432}} \Lambda(d\alpha)  \Lambda(d\beta)\\
&=O(1).
\label{PV10}}
The same reasoning can be adjusted to show that 
\besn{
&\left|\int_{-1}^0\iint_{Q_\lambda\times Q_\lambda}t\cov(\bone[-1\le  \Delta_\alpha<t<0],\bone[-1\le  \Delta_\beta<t<0])dt\Lambda(d\alpha)\Lambda(d\beta)\right|\\
&{=} O(1).
\label{PV11}}
The anticipated order of $r_5' = O\left(\lambda^{-1/2}\right)$ can be observed from \eq{PV10} and \eq{PV11}. {The proof is completed by observing that $r_i',\ 1\le i\le 5$, are all of the order $O(\lambda^{-1/2})$.}
\end{proof}

\section{Stein Couplings}\label{sec4}

The following notion was introduced by \cite{CR2010}. A triple of random variables~$(W, W', G)$ defined on the same probability space is called a \emph{Stein coupling} if 
\ben{ \label{97}
\E\clc{Gf(W')-Gf(W)} = \E\clc{ Wf(W)} 
}
for all absolutely continuous functions~$f$ with~$f(x) = \bigo(1 + |x|)$.  By taking~$f(x) = 1$, \eq{97} implies that~$\E W = 0$, and by taking~$f(w)=w$, \eq{97} implies that~$\Var W=\IE\clc{G(W'-W)}$; as usual, we assume that~$\Var W = 1$.

Now suppose~$(W,W',G)$ is a Stein coupling. Let~$\Delta = W' - W$, and let~$\mathcal{F}$ be a~$\sigma$-algebra with respect to which~$W$ is measurable.
Then 
\be{
\E \clc{Wf(W) }= \E\int_{-\infty}^\infty f'(W+t)\hat{K}(t)dt,
}
where 
\be{
\hat{K}(t)= \E\bclc{G\bclr{\bone[\Delta > t >0]-\bone[\Delta < t \le 0]}\given\mathcal{F}}.
}
One particular choice of~$\Kin$ and~$\Kout$ that turns out to be useful is
\bea{
	\hKin(t) &= \E\bclc{G\bclr{\bone[\Delta\wedge1 \geq t >0]-\bone[\Delta\vee(-1) \leq t \le 0]}\given\calF},\\
\hKout(t)&= \E\bclc{G[\bone\bclr{\Delta \geq t > \Delta\wedge1]-\bone[\Delta<t<\Delta\vee(-1)]}\given\calF},
} 
from which, via Theorem~\ref{thm1}, we can immediately deduce a bound on the normal approximation of~$W$ in terms of quantities involving only the Stein coupling. Whereas~$r_2$ and~$r_3$ are typically straightforward to bound, $r_1$, $r_4$ and~$r_5$ can be difficult to handle without further assumptions. The following result shows, however, that, by introducing additional auxiliary random variables, we can obtain upper bounds on these quantities that are far more tractable.

\begin{thm} \label{thm5} Let~$(W,W',G)$ be a Stein coupling;  set~$\Delta = W'-W$. Let~$(G',\Delta')$ be a conditionally independent copy of~$(G,\Delta)$ given~$\calF$, and let~$(G^*,\Delta^{\!*})$ be an unconditionally independent copy of~$(G,\Delta)$. Then
\bea{
\dk(\law(W),\mathcal{N}(0,1)) \le 9s_1+11s_2+ 5s_3 + 10s_4,
}
where
\bea{
	s_1 & = \bclr{\IE\bclc{\abs{GG'}(\abs{\Delta}\wedge 1) (\abs{\Delta'-\Delta^{\!*}}\wedge 2)+\abs{G(G'-G^*)}(\abs{\Delta}\wedge 1)(\abs{\Delta^{\!*}}\wedge 1)}}^{1/2},\\
	 s_2 & = \IE\clc{\abs{G}(\abs{\Delta}\wedge1)^2},\\
	s_3 & = \IE\clc{\abs{G} (\abs{\Delta}-1) \bone[\abs{\Delta} > 1]},\\
	s_4 & = \E\bclc{ \abs{GG'}(\abs{\Delta'-\Delta^{\!*}}\wedge 1)+\abs{G(G'-G^*)}(\abs{\Delta'}\wedge\abs{\Delta^{\!*}}\wedge 1)}.
}
\end{thm}

\begin{proof} Let~$\IE^\calF$ denote conditional expectation with respect to~$\calF$, and let~$\bar{x} = (x\wedge 1)\vee(-1)$. We apply Theorem~\ref{thm1}; using Lemma~\ref{lem5}$(i)$ in the last inequality,
\bea{
	r_1^2 & = \bclr{\IE\babs{\IE^\calF(G\-\bar{\Delta})-\IE(G\-\bar{\Delta})}}^2\\
	& \leq \Var\,\IE^\calF\clr{G\bar{\Delta}}
	 = \IE\bclc{\clr{\IE^\calF\clr{G\bar{\Delta}}}^2-\clr{\IE\clc{G\bar{\Delta}}}^2} \\
	& = \IE\bclc{GG'\bar{\Delta}\bar{\Delta}' - GG^*\bar{\Delta}\bar{\Delta}^{\!*}} 
	 \leq \IE\bclc{\abs{G\bar{\Delta}} \abs{G'\bar{\Delta}' - G^*\bar{\Delta}^{\!*} }} \\
	& \leq \IE\bclc{\abs{G\bar{\Delta}}\bclr{\abs{G'}(\abs{\Delta'-\Delta^{\!*}}\wedge2) + \abs{G'-G^*}(\abs{\Delta^{\!*}}\wedge 1)}} = s_1^2.
}
The expressions for~$s_2$ and~$s_3$ are easily obtained from~$r_2$ and~$r_3$.
Letting~$I_x^\pm(t) = \bone[x\wedge1 \geq t >0]-\bone[x\vee(-1) \leq t \le 0]$ and using Lemma~\ref{lem5}$(ii)$,
\bes{
	r_4 & = \E\int_{|t|\le 1}\bclr{\hKin(t)-\Kin(t)}^2dt \\
	& = \E\int_{|t|\le 1}\bclr{\E^\calF\clc{GI_t^\pm(\Delta)} -\E\clc{GI_t^\pm(\Delta)}}^2dt \\
	& = \E\int_{|t|\le 1}\bclr{\E^\calF\clc{GI_t^\pm(\Delta)}}^2 -\bclr{\E\clc{GI_t^\pm(\Delta)}}^2dt \\
	& = \E\int_{|t|\le 1}\bclr{\E^\calF\clc{GG'I_t^\pm(\Delta)I_t^\pm(\Delta')} - \E\clc{GG^*I_t^\pm(\Delta)I_t^\pm(\Delta^{\!*})}} dt \\
	& = \E\int_{|t|\le 1}\bclr{ GG'I_t^\pm(\Delta)I_t^\pm(\Delta') - GG^*I_t^\pm(\Delta)I_t^\pm(\Delta^{\!*})} dt \\
	& = \E\bclc{ GG'(\abs{\Delta}\wedge\abs{\Delta'}\wedge 1)\bone[\Delta\Delta'>0] - GG^*(\abs{\Delta}\wedge\abs{\Delta^{\!*}}\wedge 1)\bone[\Delta\Delta^{\!*}>0]} \\
	& \leq \E\bclc{ \abs{GG'}(\abs{\Delta'-\Delta^{\!*}}\wedge 1)} + \IE\bclc{\abs{G}\abs{G'-G^*}(\abs{\Delta}\wedge\abs{\Delta^{\!*}}\wedge 1)} = s_4.
}
Finally, similar to the estimate of~$r_4$, using Lemma~\ref{lem5}$(iii)$,
\bea{
	r_5^2 &= \E\int_{|t|\le 1}|t|\bclr{\hKin(t)-\Kin(t)}^2dt\\
	&= \E\int_{|t|\le 1}\abs{t}\bclr{ GG'I_t^\pm(\Delta)I_t^\pm(\Delta') - GG^*I_t^\pm(\Delta)I_t^\pm(\Delta^{\!*})} dt\\        
	& = \ahalf\E\bclc{ GG'(\abs{\Delta}\wedge\abs{\Delta'}\wedge 1)^2\bone[\Delta\Delta'>0] - GG^*(\abs{\Delta}\wedge\abs{\Delta^{\!*}}\wedge 1)^2\bone[\Delta\Delta^{\!*}>0]} \\
	& \leq \E\bclc{ \abs{GG'}(\abs{\Delta}\wedge 1)(\abs{\Delta'-\Delta^{\!*}}\wedge 1) + \ahalf\abs{G}\abs{G'-G^*}(\abs{\Delta}\wedge\abs{\Delta^{\!*}}\wedge 1)^2}\leq s_1^2.\tag*{\qedhere}
}
\end{proof}	

Many couplings in the literature, such as the exchangeable pair and the size-bias coupling, can be formulated as Stein couplings, to which Theorem \ref{thm5} can be applied. In what follows, we construct Stein couplings for local dependence and additive functionals in classical occupancy problems and apply Theorem \ref{thm5} to obtain new errors bounds, of which the former improves a result of \cite{CS2004}.

\subsection{Local dependence}\label{sec5}

Consider a sequence of centered random variables~$X_1,\dots,X_n$ which are locally dependent in the following sense. For each~$1\leq i\leq n$, there is a set~$A_i\subset\{1,\dots,n\}$ such that~$X_i$ and~$(X_j)_{j\in A^c_i}$ are independent of each other. Moreover, for each~$1\leq i\leq n$, there is a set~$B_i\subset\{1,\dots,n\}$ such that~$A_i\subset B_i$ and such that~$(X_j)_{j\in A_i}$ and~$(X_j)_{j\in B^c_i}$ are independent of each other. Let~$W = \sum_{i=1}^n X_i$,
$Y_i=\sum_{j\in A_i}X_j$ and assume that~$\Var W = 1$.

This and related types of dependency structures were extensively studied by \cite{CS2004}; in particular, it is shown that under the above assumptions and if there is~$2<\rho\leq 4$ such that~$\E\abs{X_i}^\rho+\E\abs{Y_i}^\rho\leq \theta^p$ for some~$\theta>0$ and all~$1\leq i\leq n$, then
\ben{ \label{99}
\dk(\law(W),\mathcal{N}(0,1))
\leq \clr{13+11\kappa}n\theta^{3\wedge \rho}+2.5 \theta^{p/2}\sqrt{\kappa n},
}
where 
\ben{\label{100}
	\kappa = \sup_{1\leq i\leq n}\babs{\bclc{1\leq j\leq n\,:\, B_j\cap B_i\neq \emptyset}};
}
see \cite[Theorem 2.2]{CS2004}.

We can easily reproduce~\eq{99} (up to constants) and further improve it by means of Theorem~\ref{thm5}.

\begin{thm} \label{thm6}
Let~$X_1,\dots,X_n$ be centered random variables with dependency neighborhoods~$A_i$ and~$B_i$ as described above, set~$W=\sum_{i=1}^n X_i$, and assume that $\Var(W) = 1$, and assume that there is~$\theta>0$ such that~$\IE\{\abs{X_i}^\rho\}\vee\IE\{\abs{Y_i}^\rho\}\leq \theta^\rho$ for some~$2< \rho\leq 4$ and each~$1\leq i\leq n$. Then 
\be{
\dk(\law(W),\mathcal{N}(0,1))
\leq \clr{16+67\lambda}n\theta^{3\wedge \rho}+28 \theta^{\rho/2}\sqrt{\lambda n},
}
where
\ben{\label{101}
	\lambda = \sup_{1\leq i\leq n}\babs{\bclc{1\leq j\leq n\,:\, A_j\cap B_i\neq \emptyset}}.
}
\end{thm}

\begin{re} Note that~$\lambda$ in~\eq{101} is upper bounded by~$\kappa$ in~\eq{100}, and in fact, $\lambda$ can be substantially smaller than~$\kappa$.
\end{re}

\begin{proof}[Proof of Theorem \ref{thm6}] 
Let~$I$ and~$J$ be independent random variables, uniformly distributed on set of indices $\{1,\dots,n\}$ independently of all else. Let~$W' = W - Y_I$ and~$G = -nX_I$; then~$(W,W',G)$ is a Stein coupling, and we have~$\Delta= -Y_I$. Let~$\mathcal{F}=\sigma(X_1,\dots,X_n)$, and let~$(X_1^*,\dots,X_n^*)$ be an independent copy of~$(X_1,\dots,X_n)$. Let~$G'=-nX_J$ and~$\Delta'=-Y_J$; clearly, $(G',\Delta')$ is an independent copy of~$(G,\Delta)$ given~$\cF$. Moreover, with
\be{
	(G^*,\Delta^{\!*}) = \begin{cases}
		\clr{-nX_J,-Y_J} & \text{if~$A_J\cap B_I = \emptyset$,}\\
		\clr{-nX^*_J,-Y^*_J} & \text{if~$A_J\cap B_I \neq \emptyset$;}\\
	\end{cases}
}
it is easy to see that~$(G^*,\Delta^{\!*})$ is an unconditionally independent copy of~$(G,\Delta)$. We now apply Theorem~\ref{thm5}.
To this end, recall Young's inequality; for non-negative real numbers~$a$ and~$b$ we have~$ab\leq a^p/p+ b^q/q$ whenever~$p$ and~$q$ are H\"older conjugates. 

We start by bounding~$s_2$; we have 
\be{
	s_2 = \sum_{i=1}^n \IE\bclc{\abs{X_i}(|Y_i|\wedge 1)^2}.
}
If~$3\leq\rho\leq 4$, Young's inequality yields
\be{
	\IE\bclc{\abs{X_i}(|Y_i|\wedge 1)^2} \leq \frac{\IE\abs{X_i}^3}{3} + \frac{2\,\IE\abs{Y_i}^3}{3}\leq \theta^3,
}
and if~$2<\rho<3$, Young's inequality yields
\bes{
	\IE\bclc{\abs{X_i}(|Y_i|\wedge 1)^2} 
	&\leq \frac{\IE\abs{X_i}^\rho}{\rho} + \frac{(\rho-1)\,\IE(\abs{Y_i}\wedge 1)^{2\rho/(\rho-1)}}{\rho}
	\leq \bbclr{\frac{1}{\rho}+\frac{\rho-1}{\rho}}\theta^\rho= \theta^\rho,
}
so that
\be{
	s_2\leq n\theta^{3\wedge\rho}.
}
We continue to bound~$s_3$; we have 
\be{
	s_3 = \sum_{i=1}^n \IE\bclc{\abs{X_i}(\abs{Y_i}- 1)\bone[\abs{Y_i}>1]}.
}
If~$3\leq\rho\leq 4$, Young's inequality yields
\be{
	\IE\bclc{\abs{X_i}(\abs{Y_i}- 1)\bone[\abs{Y_i}>1]} \leq \frac{\IE\abs{X_i}^3}{3} + \frac{2\,\IE\abs{Y_i}^3}{3}\leq \theta^3,
}
and if~$2<\rho<3$, Young's inequality yields
\bes{
	\IE\bclc{\abs{X_i}(\abs{Y_i}- 1)\bone[\abs{Y_i}>1]} 
	&\leq \frac{\IE\abs{X_i}^\rho}{\rho} + \frac{(\rho-1)\,\IE\bclr{\abs{Y_i}\bone[\abs{Y_i}>1]}^{\rho/(\rho-1)}}{\rho}\\
	& \leq \frac{\IE\abs{X_i}^\rho}{\rho} + \frac{(\rho-1)\,\IE\abs{Y_i}^{\rho}}{\rho}\leq\theta^\rho,
}
so that
\be{
	s_3\leq n\theta^{3\wedge\rho}.
}
Now, in order to bound~$s_4$, note that
\be{
	\IE\clc{\abs{GG'}(\abs{\Delta'-\Delta^*}\wedge1)}
	= \sum_{i=1}^n\sum_{j:A_j\cap B_i\neq \emptyset}^n \IE\clc{|X_iX_j|(\abs{Y_j-Y_j^*}\wedge1)}.
}	
Using again Young's inequality, we have for~$3\leq \rho\leq 4$,
\bes{
	\IE\clc{|X_iX_j|(\abs{Y_j-Y_j^*}\wedge1)}
	&\leq \frac{1}{3}\IE\bclc{\abs{X_i}^3+\abs{X_j}^3+4\abs{Y_j}^3+4\abs{Y_j^*}^3}
	\leq \frac{10}{3}\theta^3,
}
and for~$2 < \rho <3$,
\bes{
	\IE\clc{|X_iX_j|(\abs{Y_j-Y_j^*}\wedge1)}
	&\leq \frac{1}{\rho}\IE\abs{X_i}^\rho
	+ \frac{1}{\rho}\IE\abs{X_j}^\rho
	+ \frac{\rho-2}{\rho}\IE(\abs{Y_j-Y_j^*}\wedge1)^{\rho/(\rho-2)}\\
	&\leq \frac{1}{\rho}\IE\abs{X_i}^\rho
	+ \frac{1}{\rho}\IE\abs{X_j}^\rho
	+ \frac{\rho-2}{\rho}\IE(\abs{Y_j-Y_j^*})^{\rho} \\
	&\leq \frac{1}{\rho}\IE\abs{X_i}^\rho
	+ \frac{1}{\rho}\IE\abs{X_j}^\rho
	+ 2^{\rho-1}\frac{\rho-2}{\rho}\bclr{\IE\abs{Y_j}^\rho+\IE\abs{Y_j^*}^\rho}
	\\
	&\leq \bbbclr{\frac{2}{\rho}
	+ 2^{\rho}\frac{\rho-2}{\rho}}\theta^{\rho}
	\leq \frac{10}{3}\theta^{\rho}.
}
Similarly, 
\bes{
 \E\bclc{\abs{G(G'-G^*)}(\abs{\Delta'}\wedge\abs{\Delta^{\!*}}\wedge 1)}
 & \leq  \E\bclc{\abs{G(G'-G^*)}(\abs{\Delta'}\wedge 1)} \\
 & = \sum_{i=1}^n\sum_{j:A_j\cap B_i\neq \emptyset}^n \E\bclc{\abs{X_i(X_j-X_j^*)}(\abs{Y_j}\wedge 1)}.
}
Using again Young's inequality, we have for~$3\leq \rho\leq 4$,
\bes{
	\E\bclc{\abs{X_i(X_j-X_j^*)}(\abs{Y_j}\wedge 1)}
	&\leq \frac{1}{3}\IE\bclc{\abs{X_i}^3+4\abs{X_j}^3+4\abs{X_j^*}^3+\abs{Y_j}^3}
	\leq \frac{10}{3}\theta^3,
}
and for~$2 < \rho <3$,
\bes{
	&\E\bclc{\abs{X_i(X_j-X_j^*)}(\abs{Y_j}\wedge 1)}\\
	&\qquad\leq \frac{1}{\rho}\IE\abs{X_i}^\rho
	+ \frac{1}{\rho}\IE\abs{X_j-X_j^*}^\rho
	+ \frac{\rho-2}{\rho}\IE(\abs{Y_j}\wedge1)^{\rho/(\rho-2)}\\
	&\qquad\leq \frac{1}{\rho}\IE\abs{X_i}^\rho
	+ \frac{1}{\rho}\IE\abs{X_j-X_j^*}^\rho
	+ \frac{\rho-2}{\rho}\IE\abs{Y_j}^{\rho} \\
	&\qquad\leq \frac{1}{\rho}\IE\abs{X_i}^\rho
	+ \frac{2^{\rho-1}}{\rho}\bclr{\IE\abs{X_j}^\rho+\IE\abs{X_j^*}^\rho}
	+ \frac{\rho-2}{\rho}\IE\abs{Y_j}^\rho
	\\
	&\qquad\leq \bbbclr{\frac{1+2^{\rho}}{\rho}
	+ \frac{\rho-2}{\rho}}\theta^{\rho}
	\leq \frac{10}{3}\theta^{\rho}.
}
Hence,
\be{
	s_4\leq \frac{20}{3}n\lambda\theta^{3\wedge \rho}.
}
Finally, to bound~$s_1$, note that
\bes{
	&\IE\bclc{\abs{GG'}(\abs{\Delta}\wedge 1) (\abs{\Delta'-\Delta^{\!*}}\wedge 2)}\\
 	&\qquad = \sum_{i=1}^n\sum_{j:A_j\cap B_i\neq \emptyset}^n \E\bclc{\abs{X_iX_j}(\abs{Y_i}\wedge 1)(\abs{Y_j-Y_j^*}\wedge 2)}.
}
Using Young's inequality, we have
\bes{
&\E\bclc{\abs{2X_iX_j}(\abs{Y_i}\wedge 1)((\ahalf\abs{Y_j-Y_j^*})\wedge 1)} \\
&\qquad\leq \frac{\IE\abs{2X_i}^\rho}{\rho}
+\frac{\IE\abs{X_j}^\rho}{\rho}\\
&\kern8em+\frac{(\rho-2)\bclr{\IE(\abs{Y_i}\wedge 1)^{2\rho/(\rho-2)}+\IE((\ahalf\abs{Y_j-Y_j^*})\wedge 1)^{2\rho/(\rho-2)}}}{2\rho}\\
&\qquad\leq \frac{\IE\abs{2X_i}^\rho}{\rho}
+\frac{\IE\abs{X_j}^\rho}{\rho}
+\frac{(\rho-2)\bclr{\IE\abs{Y_i}^\rho+2^{-\rho}\IE\abs{Y_j-Y_j^*}^\rho}}{2\rho}\\
&\qquad\leq \bbbclr{\frac{1+2^\rho}{\rho}+\frac{(\rho-2)(1+2^{-\rho+\rho-1+1})}{2\rho}}\theta^\rho \leq \frac{19}{4}\theta^\rho.
}
Similarly,
\bes{
	&\IE\bclc{\abs{G(G'-G^*)}(\abs{\Delta}\wedge 1)(\abs{\Delta^{\!*}}\wedge 1)}\\
 	&\qquad = \sum_{i=1}^n\sum_{j:A_j\cap B_i\neq \emptyset}^n \E\bclc{\abs{X_i(X_j-X_j^*)}(\abs{Y_i}\wedge 1)(\abs{Y_j^*}\wedge 1)}.
}
Using Young's inequality another time, we have
\bes{
&\E\bclc{\abs{X_i(X_j-X_j^*)}(\abs{Y_i}\wedge 1)(\abs{Y_j^*}\wedge 1)} \\
&\qquad\leq \frac{\IE\abs{X_i}^\rho}{\rho}
+\frac{\IE\abs{X_j-X_j^*}^\rho}{\rho}
+\frac{(\rho-2)\bclr{\IE(\abs{Y_i}\wedge 1)^{2\rho/(\rho-2)}+\IE(\abs{Y_j^*}\wedge 1)^{2\rho/(\rho-2)}}}{2\rho}\\
&\qquad\leq \frac{\IE\abs{X_i}^\rho}{\rho}
+\frac{\IE\abs{X_j-X_j^*}^\rho}{\rho}
+\frac{(\rho-2)\bclr{\IE\abs{Y_i}^\rho+\IE\abs{Y_j^*}^\rho}}{2\rho}\\
&\qquad\leq \bbbclr{\frac{1+2^\rho}{\rho}+\frac{\rho-2}{\rho}}\theta^\rho \leq \frac{19}{4}\theta^\rho.
}
Thus,
\be{
	s_1^2\leq \frac{19}{2}n\lambda\theta^{\rho}.
}
Combining the bounds and applying Theorem~\ref{thm5} yields the final bound.
\end{proof}

{

\def\phi{\varphi}

\subsection{Additive functionals in the classical occupancy scheme}

Consider the following multinomial urn model. A total of~$m$ balls are independently distributed among~$n$ urns in such a way that a ball is placed in urn~$i$ with probability~$p_i$, where~$\sum_{i=1}^m p_i = 1$. Let~$\xi_i$ be the number of balls  urn~$i$ contains after the balls have been distributed. For each~$1\leq i \leq n$, let~$\phi_i$ be a real-valued function on the non-negative integers. We are interested in the statistic
\be{
	V = \sum_{i=1}^n \phi_i(\xi_i),
}
respectively, the centered and normalized version 
\be{
	W = \frac{1}{\sigma}\sum_{i=1}^n \clr{\phi_i(\xi_i)-\mu_i},
}
where~$\mu_i = \IE\phi_i(\xi_i)$ and~$\sigma^2 = \Var V$. In the special case where~$p_i=1/n$ and~$\phi_i=\phi$, this statistic has been studied by various authors; in particular, \cite{Dembo1996} used Stein's method and size-biasing to obtain error bounds, but only for smooth probability metrics (we refer to their paper for general references). We give the corresponding result for the Kolmogorov distance for general~$p_i$ and~$\phi_i$, which is, to the best of our knowledge, new in this generality.

\begin{thm}\label{thm9} Let~$m$, $n$ and $W\!$ be as in the preceding paragraph. Assume there are positive constants~$K_1$, $K_2$ and~$K_3$ such that
\be{
	\abs{\phi_i(x)}\leq K_1e^{K_1x}, \qquad x\geq 0,\quad 1\leq i\leq n,
}
and such that
\ben{\label{102a}
	\sup_{1\leq i\leq n} p_i \leq \frac{K_2}{m},\qquad\text{and}
	\qquad n\leq K_3m.
}
Then there is a constant~$C:=C(K_1,K_2,K_3)$ such that
\ben{\label{103}
	 \dk(\law(W),\mathcal{N}(0,1)) \leq C\bbclr{\frac{n^{1/2}}{\sigma^2} + \frac{n}{\sigma^3}}.
}
\end{thm}

Specializing to the case originally considered by \cite{Dembo1996}, we have that~$\sigma^2 \asymp n$ as long as~$\phi$ is not a linear function (see \cite[Remark 3.1]{Dembo1996}), and we have the following corollary. 
 
\begin{cor}\label{cor10} If~$\phi_i = \phi$ for all~$i$ for some non-linear function~$\phi$, if~$p_i = 1/n$, and if~$0<\lim_{n\to\infty} n/m<\infty$, then~$\sigma^2\asymp n$ and
thus,
\be{
	 \dk(\law(W),\mathcal{N}(0,1)) = \bigo\bbclr{\frac1{n^{1/2}}}.
}
\end{cor}

\begin{proof}[Proof of Theorem~\ref{thm9}] In what follows the reference set in expressions like `$i\neq I$' or `$i\not\in A$' is~$\{1,\dots,n\}$, so that these expressions have to be read as `$i\in \{1,\dots,n\}\setminus\{I\}$' or `$i\in \{1,\dots,n\}\setminus A$', respectively. Moreover, we use the convention that sums of the form `$\sum_{i\neq j}$' stand for single sums over the first variable, not double sums over both variables. 

\smallskip\noindent\textbf{Stein coupling.}\enskip Denote by~$(\xi_i)_{1\leq i\leq n}$ the ball counts in the respective urns, and let~$\calF=\sigma(\xi_1,\dots,\xi_n)$. Let~$I$ be uniformly distributed on~$\{1,\dots,n\}$, independently of all else. Given~$I$, let~$\iota_1,\iota_2,\dots$ be an i.i.d.\ sequence, where~$\IP[\iota_1=i|I] = p_i/(1-p_I)$ for all~$i\neq I$, let~$\eta_i = \xi_i + \sum_{k=1}^{\xi_I}\bone[\iota_k=i]$ for~$i\neq I$, and let~$N \coloneqq  \{\iota_1,\dots,\iota_{\xi_I}\}$. The family of random variables~$(\eta_i)_{i\neq I}$ represents the configuration of balls in the urns if the balls from the~$I$th urn are redistributed among the other urns, and~$N$ is the set of urns having received at least one ball during that redistribution.

Let~$W$ be defined as before, and let 
\be{
	G=-\frac{n}{\sigma}(\phi_I(\xi_I)-\mu_I),\qquad W' = \frac{1}{\sigma}\bbclr{-\mu_I+\sum_{i\neq I} \bclr{\phi_i(\eta_i)-\mu_i}}.
}	
It is not difficult to see that~$(W,W',G)$ is a Stein coupling and that 
\be{
	\Delta = -\frac{1}{\sigma}\bbclr{\phi_I(\xi_I)+\sum_{i\in N}\bclr{\phi_i(\xi_i)-\phi_i(\eta_i)}};
}
see, for example, \cite[Construction~2A]{CR2010}.

\smallskip\noindent\textbf{Construction of~$\boldsymbol{(G',\Delta')}$.}
Let~$\calF = \sigma\bclr{\xi_1,\dots,\xi_n}$, and let~$J$ be uniformly distributed on~$\{1,\dots,n\}$, independently of all else. Given~$J$, let~$\iota'_1,\iota'_2,\dots$ be an i.i.d.\ sequence, where~$\IP[\iota'_1=i|J] = p_i/(1-p_J)$ for all~$i\neq J$, and let~$\eta'_i = \xi_i + \sum_{k=1}^{\xi_J}\bone[\iota'_k=i]$ for~$i\neq J$, and let~$N' \coloneqq  \{\iota'_1,\dots,\iota'_{\xi_J}\}$. Define
\be{
	G'=-\frac{n}{\sigma}(\phi_J(\xi_J)-\mu_J),\qquad \Delta' = -\frac{1}{\sigma}\bbclr{\phi_J(\xi_J)+\sum_{i\in N'}\bclr{\phi_i(\xi_i)-\phi_i(\eta_i')}}
}	
Clearly, $(G',\Delta')$ is an independent copy of~$(G,\Delta)$ conditionally on~$\calF$. 

\smallskip\noindent\textbf{Construction of~$\boldsymbol{(G^*,\Delta^{\!*})}$.}
We first construct a realization~$(\xi^*_i)_{1\leq i\leq n}$ of the urn process which is independent of~$(G,\Delta)$ but which is still closely coupled to~$(\xi_i)_{1\leq i\leq n}$. To this end, let~$(\xi^\bullet_i)_{1\leq i\leq n}$ be an independent copy of~$(\xi_i)_{1\leq i\leq n}$. Set
\be{
	\xi^*_i = \xi^\bullet_i\qquad \text{for all~$i\in N\cup \{I\}$.}
}
Now, let~$\chi = \sum_{i\in N\cup\{I\}}\xi_i$ and~$\chi^\bullet = \sum_{i\in N\cup\{I\}}\xi^\bullet_i$; we distinguish three cases. 
\begin{itemize}
\item[$(i)$]
If~$\chi^\bullet=\chi$, set 
\be{
	\xi^*_i = \xi_i\qquad \text{for all~$i\not\in N\cup \{I\}$.}
}
\item[$(ii)$] If~$\chi^\bullet < \chi$, let~$\iota_1^*,\dots,\iota^*_{\chi-\chi^\bullet}$ be i.i.d.\ random variables on~$\{1,\dots,n\}\setminus(N\cup\{I\})$ with distribution given by~$\IP[\iota_1^*=k|I,N] = p_k/(1-\sum_{j\in N\cup\{I\}}p_j)$ for~$k\not\in N\cup\{I\}$, and set
\be{
	\xi^*_i = \xi_i + \sum_{j=1}^{\chi-\chi^\bullet} \bone[\iota^*_j=i]
	\qquad \text{for all~$i\not\in N\cup \{I\}$.}
}
\item[$(iii)$] If~$\chi^\bullet > \chi$, let~$\iota_1^*,\dots,\iota_{\chi^\bullet - \chi}$ be constructed recursively as follows. Let~$\iota^*_1$ have distribution given by~$\IP[\iota^*_1=k|I,(\xi_i)_{i\not\in N\cup\{I\}}] = \xi_k / (m-\chi)$ for~$k\not\in N\cup\{I\}$. For~$1<l<\chi^\bullet - \chi$, assume~$\iota^*_1,\dots,\iota^*_l$ have been sampled, and let~$\iota^*_{l+1}$ have distribution given by~$\IP[\iota^*_{l+1}=k|I,(\xi_i)_{i\not\in N\cup\{I\}},(\iota^*_i)_{1\leq i\leq l}] = \bclr{\xi_k - \sum_{j=1}^l \bone[\iota^*_j=k]} / (m-\chi - l)$ for~$k\not\in N\cup\{I\}$. Finally, set
\be{
	\xi^*_i = \xi_i - \sum_{j=1}^{\chi^\bullet-\chi} \bone[\iota^*_j=i]
	\qquad \text{for all~$i\not\in N\cup \{I\}$.}
}
\end{itemize}
It is not difficult to check that the distribution of~$(\xi^*)_{1\leq i\leq n}$ is the same regardless of~$(\xi_i)_{i\in N\cup\{I\}}$ (the key to this is the observation that~$(\xi_i)_{i\not\in N\cup\{I\}}$ given~$(\xi_i)_{i\not\in N\cup\{I\}}$ is like distributing~$n-\chi$ balls independently among the urns~$i\not\in N\cup\{I\}$ proportionally to their respective probabilities~$(p_i)_{i\not\in N\cup\{I\} }$), so that~$(\xi^*)_{1\leq i\leq n}$ is independent of~$(\xi_i)_{i\in N\cup\{I\}}$.

With~$\iota'_1,\iota'_2,\dots$ as before, we set~$\eta^*_i = \xi^*_i + \sum_{k=1}^{\xi^*_J}\bone[\iota'_k=i]$ and~$N^* = \{\iota'_1,\dots,\iota'_{\xi^*_J}\}$ and define
\be{
	G^*=-\frac{n}{\sigma}(\phi_J(\xi^*_J)-\mu_J),\qquad \Delta^{\!*} = -\frac{1}{\sigma}\bbclr{\phi_J(\xi^*_J)+\sum_{i\in N^*}\bclr{\phi_i(\xi^*_i)-\phi_i(\eta_i^*)}}.
}	
Since~$(\xi^*_i)_{1\leq i\leq n}$, $J$ and~$(\iota'_1,\iota'_2,\dots)$
are all independent of~$(\xi_i)_{i\in N\cup\{I\}}$, it follows that~$(G^*,\Delta^{\!*})$ is an independent copy of~$(G,\Delta)$.

\smallskip\noindent\textbf{Bounding the error terms ---  preliminaries.} We first show that we can assume without loss of generality that
\ben{\label{104}
	\sup_{1\leq i\leq n}p_i \leq \frac{K_2}{m}\wedge\frac{1}{2}.
}
Indeed, there are only finitely many $m$ such that $K_2/m\geq 1/2$, that is, such that $m\leq 2K_2$, and hence, using the second condition of \eq{102a}, there are only finitely many pairs $(m,n)$ such that $K_2/m\geq 1/2$. For these finitely many cases, we can choose~$C$ large enough to make~\eq{103} true whenever $\sigma^2>0$.

We will use~\eq{104} repeatedly to conclude that, for example,
\be{
	\frac{p_j}{1-p_i} \leq \frac{2K_2}{m}\leq \frac{C}{m}.
}
Throughout the proof we will use~$C$ to denote a constant that can change from expression to expression but that only depends on~$K_1$, $K_2$ and~$K_3$. 

\smallskip\noindent\textbf{Bounding the error terms --- an event of small probability.} Define the event
\be{
	B_1 = \bclc{\text{$\xi_J = \xi_J^*$}} \cap \bclc{\text{$\xi_{\iota'_j} = \xi_{\iota'_j}^*$ for all~$1\leq j\leq \xi_J$}}
} 
and note that
\be{
	B_1\subset \{G'=G^*\}, \qquad B_1\subset \{\Delta'=\Delta^{\!*}\}.
}
Let~$N^\bullet=\{\iota_1^*,\dots,\iota_{\abs{\chi^\bullet-\chi}}^*\}=\{i\not\in N\cup\{I\}\,:\,\xi_i\neq \xi_i^*\}$, let~$M=N^\bullet\cup N\cup\{I\}$ (disjoint union!), and define the event
\be{
	B_2 = \bclc{M\cap \{J\}\cap N'=\emptyset}.
}
Clearly, $B_2\subset B_1$; thus, setting~$A = B_2^c$, we have
\be{
	\{G'\neq G^*\}\subset A, \qquad \{\Delta'\neq \Delta^{\!*}\}\subset A.
}
We proceed to bound~$\IP[A]$. Let~$\calG=\sigma\bclr{I,(\iota_i)_{i\geq 1},(\iota^*_i)_{i\geq 1},(\xi_i)_{1\leq i\leq n},(\xi^\bullet_i)_{1\leq i\leq n}}$; we have
\bes{
	\IP\cls{A\given\calG}
	& = \frac{1}{n}\sum_{j=1}^n \IP\cls{A\given J=j,\calG} 
	 \leq \frac{\abs{M}}{n} + \frac{1}{n}\sum_{j\not\in M}\IE\{\abs{M\cap N'}\,|J=j,\calG\}  \\
	& \leq \frac{\abs{M}}{n} + \frac{1}{n}\sum_{j\not\in M}\IE\bbbclc{\sum_{k=1}^{\xi_J}\bone[\iota'_k\in M]\given J=j,\calG}  
	 = \frac{\abs{M}}{n} + \frac{1}{n}\sum_{j\not\in M}\xi_j \sum_{k\in M}\frac{p_k}{1-p_j} \\
	& \leq  \frac{\abs{M}}{n} +\frac{2K_2\abs{M}}{nm}\sum_{j\not\in M}\xi_j.
}
Since~$\sum_{j\not\in M}\xi_j\leq m$, and since 
\be{
	\abs{M} = 1+\abs{N}+\abs{\chi-\chi^\bullet}\leq C(1 + \chi+\chi^\bullet),
}
we obtain
\ben{\label{105}
\IP\cls{A\given I,(\iota_i)_{i\geq 1},(\iota^*_i)_{i\geq 1},(\xi_i)_{1\leq i\leq n},(\xi^\bullet_i)_{1\leq i\leq n} }	\leq \frac{C}{n}\clr{1+\chi+\chi^\bullet}.
}
Let now~$\calG=\sigma\bclr{I,(\iota_i)_{i\geq 1},\xi_I}$. We have
\bes{
	&\IE\bbbclc{\sum_{i\in N\cup\{I\}}(\xi_i+\xi_i^\bullet)\given \calG} 
	 \leq \IE\bbbclc{\sum_{i\neq I}(\xi_i+\xi_i^\bullet)\sum_{k=1}^{\xi_I}\bone[\iota_k=i]\given \calG}\\
	& \qquad = \IE\bbbclc{\sum_{i\neq I}(\xi_i+\xi_i^\bullet)\sum_{k=1}^{\xi_I}\frac{p_i}{1-p_I}\given \calG}
	 \leq \frac{2K_2\xi_I}{m}\IE\bbbclc{\sum_{i\neq I}(\xi_i+\xi_i^\bullet)\given \calG} 
	 \leq 4K_2\xi_I,
}
which, together with~\eq{105}, implies that
\be{
	\IP\cls{A\given I,(\iota_i)_{i\geq 1},\xi_I}	\leq \frac{C(\xi_I+1)}{n}.
}
Taking expectation and applying Lemma~\ref{lem3}, we obtain
\be{
	\IP[A]\leq \frac{C}{n}.
}	
\smallskip\noindent\textbf{Bounding the error terms ---~$\boldsymbol{G}$ and~$\boldsymbol{\Delta}$.} Note first that
\be{
	\abs{\Delta}^3 \leq \frac{C}{\sigma^3}\bbbclr{e^{K_1\xi_I} + 2\sum_{i\in N}e^{K_1(\xi_i+\xi_I)}}^3 
	\leq \frac{C}{\sigma^3}\bbbclr{e^{3K_1\xi_I} + \abs{N}^{2}\sum_{i\in N}e^{3K_1(\xi_i+\xi_I)}},
}
and since~$\abs{N}\leq \xi_I\leq e^{\xi_I}$, we can further bound this by
\be{
	\abs{\Delta}^3 \leq \frac{C}{\sigma^3}\sum_{i\in N\cup\{I\}}e^{C(\xi_i+\xi_I)}.
}
Now, given~$I$ and~$\xi_I$, the family of random variables~$(\xi_i)_{i\neq I}$ is again an urn model; in particular, we have
\be{
	\law\clr{\xi_i\given I,\xi_I} 
		= \Bi\bbclr{m-\xi_I,\frac{p_i}{1-p_I}}.
}
Since the set~$N$ is chosen independently of~$(\xi_i)_{i\neq I}$, we can apply Lemma~\ref{lem3} consecutively, and we obtain
\be{
	\IE{ \sum_{i\in N\cup\{I\}}e^{C (\xi_i+\xi_I)}} 
	\leq C\IE\clc{ (1+\abs{N})e^{C\xi_I}}
	\leq C\IE\clc{ e^{C\xi_I}}
	\leq C. 
}
Thus,
\be{
	\IE\abs{\Delta}^3 \leq \frac{C}{\sigma^3}.
}
Now,
\be{
	\abs{G}^3 \leq \frac{Cn^3}{\sigma^3}\bclr{e^{3K_1\xi_I}+\mu_I^3} 
}
and since~$\mu^3_i\leq K_1\IE e^{3K_1\xi_i}$, we have 
\be{
	\IE \abs{G}^3 \leq \frac{Cn^3}{\sigma^3}\IE e^{C\xi_I}\leq\frac{Cn^3}{\sigma^3}.
}

\smallskip\noindent\textbf{Bounding the error terms ---~$\boldsymbol{G\bone_A}$ and~$\boldsymbol{\Delta \bone_A}$.} 
Note first that
\be{
	\Delta^4 \leq \frac{C}{\sigma^4}\bbbclr{e^{K_1\xi_I} + 2\sum_{i\in N}e^{K_1(\xi_i+\xi_I)}}^4 
	\leq \frac{C}{\sigma^4}\bbbclr{e^{4K_1\xi_I} + \abs{N}^{3}\sum_{i\in N}e^{4K_1(\xi_i+\xi_I)}},
}
and since~$\abs{N}\leq \xi_I$, we can further bound this by
\be{
	\Delta^4 \leq \frac{C}{\sigma^4}\sum_{i\in N\cup\{I\}}e^{C(\xi_i+\xi_I)}.
}
Thus, using~\eq{105}, we obtain that
\bes{
	\IE\clc{\Delta^4 \bone_A}
	&\leq \frac{C}{n\sigma^4}\IE\bbbclc{\bbbclr{ \sum_{i\in N\cup\{I\}}e^{C(\xi_i+\xi_I)}}\bbbclr{1+\sum_{i\in N\cup\{I\}}(\xi_i+\xi_i^\bullet)}} \\
	&\leq \frac{C}{n\sigma^4}\IE\bbbclc{\bbbclr{ \sum_{i\in N\cup\{I\}}e^{C(\xi_i+\xi_I)}}\bbbclr{\sum_{i\in N\cup\{I\}}e^{C(\xi_i+\xi_i^\bullet)}}} \\
	&\leq \frac{C}{n\sigma^4}\IE\bbbclr{ \sum_{i\in N\cup\{I\}}e^{C(\xi_i+\xi_i^\bullet+\xi_I)}}^2 
	\leq \frac{C}{n\sigma^4}\IE{ \sum_{i\in N\cup\{I\}}e^{C(\xi_i+\xi_i^\bullet+\xi_I)}}.
}
Similarly as before,
\ben{\label{106}
	\IE{ \sum_{i\in N\cup\{I\}}e^{C(\xi_i+\xi_i^\bullet+\xi_I)}} 
	\leq C\IE\bclc{ (1+\abs{N})e^{C\xi_I}}
	\leq C. 
}
Thus,
\be{
	\IE\clc{\Delta^4 \bone_A} \leq \frac{C}{n\sigma^4}.
}
In very much the same way, we deduce that
\be{
	\IE\bclc{G^4 \bone_A} \leq \frac{Cn^3}{\sigma^4}.
}
\smallskip\noindent\textbf{Bounding the error terms ---~$\boldsymbol{G'\bone_A}$ and~$\boldsymbol{\Delta' \bone_A}$.} 
Similarly as for~$\Delta^4$,
\be{
	\Delta'^4 
	\leq \frac{C}{\sigma^4}\sum_{i\in N'\cup\{J\}}e^{C(\xi_i+\xi_J)}.
}	
Let us now prove first that
\besn{\label{107}
&\IP[A|J,(\iota'_i)_{1\leq i\leq n},(\xi_i)_{1\leq i\leq n},(\xi^\bullet_i)_{1\leq i\leq n}] \\
	&\enskip\qquad \leq \frac{C}{n}\bbbclr{1+\xi_J + \frac{1}{m}\bbbclr{1+\sum_{i\in N'\cup\{J\}}\xi_i}\sum_{i\not\in N'\cup\{J\}}(\xi_i+\xi_i^\bullet)^2}.
}
To this end, let~$\calG=\sigma\bclr{J,(\iota'_i)_{1\leq i\leq n},(\xi_i)_{1\leq i\leq n},(\xi^\bullet_i)_{1\leq i\leq n}}$. We have
\bes{
	\IP[A|\calG] 
	& = \frac{1}{n}\sum_{i=1}^n\IP[A|I=i,\calG] 
	 = \frac{1+\abs{N'}}{n} + \frac{1}{n}\sum_{i\not\in N'\cup\{J\}}\IP[A|I=i,\calG] \\
	& \leq \frac{1+\xi_J}{n} + \frac{1}{n}\sum_{i\not\in N'\cup\{J\}}\IE\bclc{\abs{M\cap (N'\cup\{J\})}\given I=i,\calG} \\
	& \leq \frac{1+\xi_J}{n} + \frac{1}{n}\sum_{i\not\in N'\cup\{J\}}\IE\bclc{\abs{N\cap (N'\cup\{J\})}+\abs{N^\bullet\cap (N'\cup\{J\})}\given I=i,\calG}.
}
Now, for~$i\not\in N'\cup\{J\}$,
\bes{
	&\IE\bclc{\abs{N\cap (N'\cup\{J\})}\given I=i,\calG}  \\
	&\qquad \leq  \IE\bbbclc{\sum_{k=1}^{\xi_I}\bone[\iota_k\in N'\cup\{J\}]\given I=i,\calG} 
	 = \xi_i \sum_{k\in N'\cup\{J\}}\frac{p_k}{1-p_i} 
	 \leq \frac{C\xi_i(1+\xi_J)}{m},
}
and
\bml{
	\IE\bclc{\abs{N^\bullet\cap (N'\cup\{J\})}\given I=i,\calG} \\
	\shoveleft\qquad\leq \IE\bbbclc{\sum_{k=1}^{\abs{\chi-\chi^\bullet}}\bone[\iota^*_k\in N'\cup\{J\}]\given I=i,\calG} \\
	\shoveleft\qquad\leq \IE\bbbclc{\sum_{k=1}^{\abs{\chi-\chi^\bullet}}\bbbclr{\sum_{l\in N'\cup\{J\}}\frac{p_l}{1-\sum_{j\in N\cup\{I\}}p_j}+\frac{\sum_{l\in N'\cup\{J\}}\xi_l}{m-\chi-k+1}}\given I=i,\calG} \\
	\shoveleft\qquad \leq \IE\bbbclc{\sum_{k=1}^{\abs{\chi-\chi^\bullet}}\bbbclr{(1+\xi_J)\bone\bbcls{\xi_i+1\leq \frac{m}{2K_2}}\frac{2K_2}{m}+\bone\bbcls{\xi_i+1> \frac{m}{2K_2}}\\
	+\frac{\sum_{l\in N'\cup\{J\}}\xi_l}{m-\chi-k+1}}\given I=i,\calG}\hfilneg\hskip\multlinegap \\
	\shoveleft\qquad \leq \frac{C(\xi_i+\xi_J+1)}{m}\IE\bclc{\abs{\chi-\chi^\bullet}\given I=i,\calG}\\
	\qquad\qquad\quad+\bbbclr{\sum_{k\in N'\cup\{J\}}\xi_k}\IE\bbbclc{\sum_{k=1}^{\abs{\chi-\chi^\bullet}}\frac{1}{m-\chi-k+1}\given I=i,\calG}.\hfill
}
First observe that
\bes{
	& \IE\bclc{\abs{\chi-\chi^\bullet}\given I=i,\calG}
	 \leq \IE\clc{\chi+\chi^\bullet\given I=i,\calG} 
	= \IE\bbbclc{\sum_{k\in N\cup\{I\}}(\xi_k+\xi_k^\bullet)\given I=i,\calG} \\
	& \qquad = \xi_i+\xi_i^\bullet+\sum_{k\neq i}(\xi_k+\xi_k^\bullet)\IE\bbbclc{\sum_{l=1}^{\xi_i}\bone[\iota_l=k] \given I=i,\calG} \\
	&\qquad =  \xi_i+\xi_i^\bullet+\sum_{k\neq i}(\xi_k+\xi_k^\bullet)\sum_{l=1}^{\xi_i}\frac{p_k}{1-p_i}
	 \leq  \xi_i+\xi_i^\bullet+\sum_{k\neq i}(\xi_k+\xi_k^\bullet)\sum_{l=1}^{\xi_i}\frac{2K_2}{m}  \\
	&\qquad \leq C(\xi_i+\xi_i^\bullet)
}
and 
\bml{
	 \IE\bclc{\clr{\chi+\chi^\bullet}^2\given I=i,\calG}\\
	 \shoveleft\qquad\leq \IE\bbbclc{\sum_{j\in N\cup\{I\}}\sum_{k\in N\cup\{I\}}(\xi_j+\xi_j^\bullet)(\xi_k+\xi_k^\bullet)\given I=i,\calG} \\
	 \shoveleft\qquad =  (\xi_i+\xi_i^\bullet)^2+2(\xi_i+\xi_i^\bullet)\IE\bbbclc{\sum_{j\in N}(\xi_j+\xi_j^\bullet)\\+\sum_{j\in N}\sum_{k\in N}(\xi_j+\xi_j^\bullet)(\xi_k+\xi_k^\bullet)\given I=i,\calG}.
}
Now,
\bes{  
	&\IE\bbbclc{\sum_{j\in N}(\xi_j+\xi_j^\bullet)\given I=i,\calG} 
	 \leq \IE\bbbclc{\sum_{j\neq I}(\xi_j+\xi_j^\bullet)\bbbclr{\sum_{l=1}^{\xi_I}\bone[\iota_l=j]}\given I=i,\calG} \\
	&\qquad = \sum_{j\neq i}(\xi_j+\xi_j^\bullet)\IE\bbbclc{\sum_{l=1}^{\xi_I}\bone[\iota_l=j]\given I=i,\calG} 
	 = \sum_{j\neq i}(\xi_j+\xi_j^\bullet)\sum_{l=1}^{\xi_i}\frac{p_j}{1-p_i}\\
	&\qquad = \sum_{j\neq i}(\xi_j+\xi_j^\bullet)\sum_{l=1}^{\xi_i}\frac{2K_2}{m}
 \leq  2K_2\xi_i,
}
and
\bml{
	\IE\bbbclc{\sum_{j\in N}\sum_{k\in N}(\xi_j+\xi_j^\bullet)(\xi_k+\xi_k^\bullet)\given I=i,\calG} \\
	\qquad \leq \IE\bbbclc{\sum_{j\neq I}\sum_{k\neq I}(\xi_j+\xi_j^\bullet)(\xi_k+\xi_k^\bullet)\sum_{l=1}^{\xi_I}\bone[\iota_l=j]\sum_{u=1}^{\xi_I}\bone[\iota_u=k]\given I=i,\calG} \\
	\qquad = \sum_{j\neq i}\sum_{k\neq i}(\xi_j+\xi_j^\bullet)(\xi_k+\xi_k^\bullet)\IE\bbbclc{\sum_{l=1}^{\xi_I}\bone[\iota_l=j]\sum_{u=1}^{\xi_I}\bone[\iota_u=k]\given I=i,\calG} \\
	\qquad \leq \sum_{j\neq i}\sum_{k\neq i}(\xi_j+\xi_j^\bullet)(\xi_k+\xi_k^\bullet)\IE\bbbclc{\sum_{l=1}^{\xi_I}\bone[\iota_l=j]\given I=i,\calG} \\
	\times\IE\bbbclc{\sum_{u=1}^{\xi_I}\bone[\iota_u=k]\given I=i,\calG}\hfilneg\hskip\multlinegap \\
	\qquad \leq \sum_{j\neq i}\sum_{k\neq i}(\xi_j+\xi_j^\bullet)(\xi_k+\xi_k^\bullet)\bbclr{\xi_i\frac{p_j}{1-p_i}}^2\\
	\qquad \leq \sum_{j\neq i}\sum_{k\neq i}(\xi_j+\xi_j^\bullet)(\xi_k+\xi_k^\bullet)\bbclr{\xi_i\frac{2K_2}{m}}^2 	
	\leq 16K_2^2\xi_i^2,\hfill
}
so that
\bes{
	\IE\bclc{\clr{\chi+\chi^\bullet}^2\given I=i,\calG}
	&  \leq  (\xi_i+\xi_i^\bullet)^2+2(\xi_i+\xi_i^\bullet)(2K_2\xi_i+16K_2^2\xi_i^2) \\
	& \leq C(\xi_i+\xi_i^{\bullet})^2.
}
Now, since~$a+\abs{a-b}\leq 2(a+b)$ for non-negative numbers~$a$ and~$b$, we have
\bes{
&\IE\bbbclc{\sum_{k=1}^{\abs{\chi-\chi^\bullet}}\frac{1}{m-\chi-k+1}\given I=i,\calG} \\
&\qquad \leq \IE\bbbclc{\frac{\abs{\chi-\chi^\bullet}}{m-m/2}+2\log(m+1)\bone[2(\chi+\chi^\bullet)>m/2]\given I=i,\calG} \\
&\qquad \leq \frac{C}{m}\bclr{\IE\bclc{\abs{\chi-\chi^\bullet}\given I=i,\calG }\IE\bclc{\clr{\chi+\chi^\bullet}^2\given I=i,\calG }}\\
&\qquad\leq \frac{C(\xi_i+\xi_i^\bullet)^2}{m}.
}
This leads to
\bes{
	&\IE\clc{\abs{N^\bullet\cap (N'\cup\{J\})}\given I=i,\calG} \\
	&\qquad \leq C\bbclr{\frac{\xi_J+1}{m}+\frac{\xi_i+1}{m}}(\xi_i+\xi_i^\bullet)+C \bbbclr{\sum_{k\in N'\cup\{J\}}\xi_k}
	\frac{1}{m}(\xi_i+\xi_i^\bullet)^2 \\
	&\qquad \leq \frac{C}{m}\bbbcls{\clr{\xi_J+\xi_i+1}(\xi_i+\xi_i^\bullet)+\bbbclr{\sum_{k\in N'\cup\{J\}}\xi_k}
	(\xi_i+\xi_i^\bullet)^2} \\
	&\qquad \leq \frac{C}{m}\bbbcls{\xi_J(\xi_i+\xi_i^\bullet)+(\xi_i+\xi_i^\bullet)^2+\bbbclr{\sum_{k\in N'\cup\{J\}}\xi_k}
	(\xi_i+\xi_i^\bullet)^2} \\
	&\qquad \leq \frac{C}{m}\bbbclr{1+\sum_{k\in N'\cup\{J\}}\xi_k}(\xi_i+\xi_i^\bullet)^2. 
}
Putting all together, \eq{107} follows. Hence,
\bml{
	\IE\clc{\Delta'^4 \bone_A}\\
	\shoveleft\qquad\leq\frac{C}{n\sigma^4}
	\IE\bbbclc{\bbbclr{\sum_{i\in N'\cup\{J\}}e^{C (\xi_i+\xi_J)}}\\	\times\bbbclr{1+\xi_J + \frac{1}{m}\bbbclr{1+\sum_{i\in N'\cup\{J\}}\xi_i}\sum_{i\not\in N'\cup\{J\}}(\xi_i+\xi_i^\bullet)^2}} \hfilneg\hskip\multlinegap\\
	\shoveleft\qquad\leq\frac{C}{n\sigma^4}
	\IE\bbbclc{\bbbclr{\sum_{i\in N'\cup\{J\}}e^{C (\xi_i+\xi_J)}}	\\
		\times\bbbclr{e^{\xi_J} + \frac{1}{m}\bbbclr{\sum_{i\in N'\cup\{J\}}e^{\xi_i}}\sum_{i\not\in N'\cup\{J\}}e^{2(\xi_i+\xi_i^\bullet)}}}.
}
Now, similarly as in~\eq{106},
\be{
	\IE\bbbclc{e^{\xi_J}\sum_{i\in N'\cup\{J\}}e^{C (\xi_i+\xi_J)}}\leq C.
}
Moreover
\bes{
&\frac{1}{m}\IE\bbbclc{\bbbclr{\sum_{i\in N'\cup\{J\}}e^{C (\xi_i+\xi_J)}}\bbbclr{\sum_{i\in N'\cup\{J\}}e^{\xi_i}}\bbbclr{\sum_{i\not\in N'\cup\{J\}}e^{2(\xi_i+\xi_i^\bullet)}}}\\
&\qquad\leq \frac{1}{m}\IE\bbbclc{\bbbclr{\sum_{i\in N'\cup\{J\}}e^{C '(\xi_i+\xi_J)}}\bbbclr{\sum_{i\not\in N'\cup\{J\}}e^{2(\xi_i+\xi_i^\bullet)}}}\\
&\qquad\leq \frac{1}{m}\IE\bbbclc{\IE\bbbclr{\sum_{i\in N'\cup\{J\}}e^{C '(\xi_i+\xi_J)}\given \xi_J}\IE\bbbclr{\sum_{i\not\in N'\cup\{J\}}e^{2(\xi_i+\xi_i^\bullet)}\given \xi_J}}\\
&\qquad\leq \frac{Cn}{m} \leq C,
}
where for the second last inequality we used the fact that conditionally on~$J$ and~$\xi_J$, the family~$(\xi_i)_{i\neq J}$ is a urn model and hence negatively associated and we can therefore apply Lemma~\ref{lem4}, and where for the last inequality we proceeded in the same way as in~\eq{106}. So, putting all together we obtain
\be{
	\IE\clc{\Delta'^4 \bone_A}\leq \frac{C}{\sigma^4n}.
}
In very much the same way we deduce that
\be{
	\IE\clc{G'^4 \bone_A} \leq \frac{Cn^3}{\sigma^4}.
}

\smallskip\noindent\textbf{Bounding the error terms ---~$\boldsymbol{G^*\bone_A}$ and~$\boldsymbol{\Delta^{\!*} \bone_A}$.} 
Similarly as for~$\Delta^4$, and using that~$\xi_i^*\leq \xi_i+\xi_i^\bullet +\chi$, we have
\ben{\label{108}
	\Delta^{\!*4} 
	\leq \frac{C}{\sigma^4}\sum_{i\in N^*\cup\{J\}}e^{C(\xi_i^*+\xi_J^*)}\leq e^{C\chi}\sum_{i\in N^*\cup\{J\}}e^{C(\xi_i+\xi_i^\bullet)}.
}	
Let~$\calG=\sigma\bclr{I,(\iota_i)_{i\geq 1},(\iota^*_i)_{i\geq 1},(\xi_i)_{1\leq i\leq n},(\xi^\bullet_i)_{1\leq i\leq n}}$; we have
\besn{\label{109}
	\IE\bbbclc{\sum_{i\in N^*\cup\{J\}}e^{C (\xi_i+\xi_i^\bullet)}\bone_A \given\calG}
	= \frac{1}{n}\sum_{j=1}^n \IE\bbbclc{\sum_{i\in N^*\cup\{J\}}e^{C (\xi_i+\xi_i^\bullet)}\bone_A \given J=j,\calG}.
}
For~$j\in M$ we have
\besn{\label{110}
	& \IE\bbbclc{\sum_{i\in N^*\cup\{J\}}e^{C (\xi_i+\xi_i^\bullet)}\bone_A \given J=j,\calG} \\
	&\qquad \leq e^{C (\xi_j+\xi_j^\bullet)} + \IE\bbbclc{\sum_{i\neq J}e^{C (\xi_i+\xi_i^\bullet)}\sum_{k=1}^{\xi_J^*}\bone[\iota'_k=i] \given J=j,\calG} \\ 
		&\qquad \leq e^{C (\xi_j+\xi_j^\bullet)} + \IE\bbbclc{\sum_{i\neq J}e^{C (\xi_i+\xi_i^\bullet)}\sum_{k=1}^{\xi_J^*}\frac{p_i}{1-p_J} \given J=j,\calG} \\ 
		&\qquad \leq e^{C (\xi_j+\xi_j^\bullet)} + \frac{2K_2}{m}\sum_{i\neq j}e^{C (\xi_i+\xi_i^\bullet)}\xi_j^*.
}
For~$j\not\in M$, we have
\bes{
	& \IE\bbbclc{\sum_{i\in N^*\cup\{J\}}e^{C (\xi_i+\xi_i^\bullet)}\bone_A \given J=j,\calG} \\
	&\qquad\leq  \IE\bbbclc{\abs{M\cap N^*}\sum_{i\in N^*\cup\{J\}}e^{C (\xi_i+\xi_i^\bullet)} \given J=j,\calG}  \\
	&\qquad\leq \IE\bbbclc{\sum_{k=1}^{\xi_J^*}\bone[\iota'_k\in M]\bbbclr{e^{C (\xi_J+\xi_J^\bullet)}+\sum_{i\neq J}e^{C (\xi_i+\xi_i^\bullet)}\sum_{k=1}^{\xi_J^*}\bone[\iota'_k=i]} \given J=j,\calG}  \\
		&\qquad \leq \frac{2K_2\abs{M}}{m}\xi_j^*e^{C (\xi_j+\xi_j^\bullet)}+ \IE\bbbclc{\xi_J^*\sum_{i\in M}e^{C (\xi_i+\xi_i^\bullet)}\sum_{k=1}^{\xi_J^*}\bone[\iota'_k=i]\given J=j,\calG}\\
	&\qquad\qquad + \IE\bbbclc{\sum_{i\neq M\cup\{J\} }e^{C (\xi_i+\xi_i^\bullet)}\sum_{k=1}^{\xi_J^*}\bone[\iota'_k\in M]\sum_{k=1}^{\xi_J^*}\bone[\iota'_k=i]\given J=j,\calG}  \\
		&\qquad \leq 
		\frac{2K_2\abs{M}}{m}\xi_j^*e^{C (\xi_j+\xi_j^\bullet)}
	+ \xi_j^*\sum_{i\in M}e^{C (\xi_i+\xi_i^\bullet)}\xi_j^*\frac{2K_2}{m}\\
	&\qquad\qquad + \sum_{i\neq M\cup\{j\} }e^{C (\xi_i+\xi_i^\bullet)}\frac{2K_2\abs{M}\xi_j^*}{m}\frac{2K_2\xi_j^*}{m},
}
where in the last inequality we used that the indicators~$(\bone[\iota_k'\in M],(\bone[\iota_k'=i])$ are negatively associated whenever~$i\not\in M$. Combining~\eq{110} and~\eq{110} with~\eq{109}, we obtain
\besn{\label{111}
	&\IE\bbbclc{\sum_{i\in N^*\cup\{J\}}e^{C (\xi_i+\xi_i^\bullet)}\bone_A \given\calG} \\
	&\qquad \leq \frac{C}{n}\sum_{j=1}^n\bbbclr{e^{C (\xi_j+\xi_j^\bullet)} + \frac{1}{m}\sum_{i\neq j}\xi_j^*e^{C (\xi_i+\xi_i^\bullet)}+\frac{\abs{M}}{m}\xi_j^*e^{C (\xi_j+\xi_j^\bullet)} \\
	&\qquad\qquad+ \frac{1}{m}\sum_{i\in M}\xi_j^{*2}e^{C (\xi_i+\xi_i^\bullet)}+\frac{\abs{M}}{m^2}\sum_{i\neq M\cup\{j\} }\xi_j^{*2}e^{C (\xi_i+\xi_i^\bullet)}}.
}
Using the inequality~$1\leq \frac{1}{m}\sum_{i=1}^n\xi_i\leq \frac{1}{m}\sum_{i=1}^n e^{\xi_i}$ and~$\xi_i^*\leq \xi_i+\xi_i^\bullet +\chi$, it is not difficult to further coarsen the bound in~\eq{111} to
\bes{
	&\IE\bbbclc{\sum_{i\in N^*\cup\{J\}}e^{C (\xi_i+\xi_i^\bullet)}\bone_A \given\calG} \leq \frac{C }{nm^2}\sum_{i=1}^n\sum_{k=1}^n\sum_{j\in M}e^{C(\xi_i+\xi_j+\xi_k+\xi_i^\bullet+\xi_j^\bullet+\xi_k^\bullet+\chi)}.
}
Combining this with~\eq{108}, we have
\bes{
	\IE\bclc{\Delta^{\!*4}\bone_A} \leq \frac{C }{\sigma^4nm^2}\sum_{k=1}^n\sum_{l=1}^n\IE{\sum_{u\in M}e^{C\clr{\xi_k+\xi_l+\xi_u+\xi_k^\bullet+\xi_l^\bullet+\xi_u^\bullet+\chi}}}.
}
Now, let~$\calG=\sigma\bclr{I,(\iota_{u})_{u\geq 1},(\xi_u)_{1\leq u\leq n},(\xi_u^\bullet)_{1\leq u\leq n}}$; we have
\bes{
	& \IE\bbbclc{\sum_{u\in N^\bullet}e^{C\clr{\xi_u+\xi_u^\bullet}}\given \calG} \\
	&\qquad \leq \IE\bbbclc{\sum_{u\not\in N\cup\{I\}}e^{C \clr{\xi_u+\xi_u^\bullet}}\sum_{w=1}^{\abs{\chi-\chi^\bullet}}\bone[\iota_w^*=u]\given \calG} \\
	&\qquad \leq \IE\bbbclc{\sum_{u\not\in N\cup\{I\}}e^{C \clr{\xi_u+\xi_u^\bullet}}\sum_{w=1}^{\abs{\chi-\chi^\bullet}}\bbclr{\frac{p_u}{1-\sum_{x\in N\cup\{I\}}p_x} + \frac{\xi_u}{m-\chi-w+1}}\given \calG} \\
	&\qquad \leq \IE\bbbclc{\sum_{u\not\in N\cup\{I\}}e^{C \clr{\xi_u+\xi_u^\bullet}}\bbbcls{\abs{\chi-\chi^\bullet}\bbclr{\frac{2K_2}{m} +
	\bone\bbcls{\xi_I+1> \frac{m}{2K_2}}} \\
	&\qquad\qquad\qquad\qquad+ \frac{\xi_u\abs{\chi-\chi^\bullet}}{m-m/2} + 2\log(m+1)\bone[2\clr{\chi+\chi^\bullet}>m/2]}\given \calG}\\
	&\qquad \leq \frac{C }{m}\IE\bbbclc{\sum_{u\not\in N\cup\{I\}}e^{C \clr{\xi_u+\xi_u^\bullet}}\bbclr{\clr{\chi+\chi^\bullet}\clr{\xi_I+1} + \xi_u\clr{\chi+\chi^\bullet} + (\chi+\chi^\bullet)^2}\given \calG} \\
	&\qquad \leq \frac{C }{m}\sum_{u\not\in N\cup\{I\}}e^{C \clr{\xi_u+\xi_u^\bullet+\chi+\chi^\bullet}}.
}
Thus, since~$M$ is the disjoint union of~$N\cup\{I\}$ and~$N^\bullet$, we obtain
\bes{
	\IE\bbbclc{\sum_{u\in M}e^{C \clr{\xi_k+\xi_l+\xi_u+\xi_k^\bullet+\xi_l^\bullet+\xi_u^\bullet+\chi}}} 
	&\leq\frac{C }{m}\IE\bbbclc{\sum_{u\not\in N\cup\{I\}}e^{C \clr{\xi_k+\xi_l+\xi_u+\xi_k^\bullet+\xi_l^\bullet+\xi_u^\bullet+\chi+\chi^\bullet}}} \\
	&\qquad + \IE\bbbclc{\sum_{u\in N\cup\{I\}}e^{C \clr{\xi_k+\xi_l+\xi_u+\xi_k^\bullet+\xi_l^\bullet+\xi_u^\bullet+\chi}}}.
}
Now, conditioning on~$I$, $\xi_I$ and~$(\iota_i)_{i\geq 1}$, we can apply Lemma~\ref{lem4} and then Lemma~\ref{lem3}
so that, for example,
\be{
	\IE \bclc{e^{C \clr{\xi_k+\xi_l+\xi_u+\xi_k^\bullet+\xi_l^\bullet+\xi_u^\bullet+\chi}}\given I,\xi_I, (\iota_i)_{i\geq 1}}\leq C\cdot C^{\xi_I}\cdot e^{C\xi_I} \leq Ce^{C\xi_I}.
}
Hence,
\bes{
	\IE\bbbclc{\sum_{u\in M}e^{C \clr{\xi_k+\xi_l+\xi_u+\xi_k^\bullet+\xi_l^\bullet+\xi_u^\bullet+\chi}}} 
	&\leq\frac{C }{m}\IE\bbbclc{\sum_{u\not\in N\cup\{I\}}e^{C \xi_I}}+ C\IE\bbbclc{\sum_{u\in N\cup\{I\}}e^{C \xi_I}} \\
	&\leq C \bbclr{\frac{n}{m}+1}\IE e^{C  \xi_I}
	\leq C,
}
which finally leads to
\be{
	\IE\clc{\Delta^{\!*4}\bone_A} 
 \leq \frac{C}{\sigma^4nm^2}\sum_{k=1}^n\sum_{l=1}^n C \leq  \frac{C }{\sigma^4n}.
}
Again, in very much the same way we can prove that
\be{
	\IE\clc{G^{*4}\bone_A} 
 \leq \frac{Cn^3}{\sigma^4}.
}

\smallskip\noindent\textbf{Combining the bounds.} Collecting the bounds we need in order to employ Lemma~\ref{rem1} below, we have
\bg{
	g_0 \leq \frac{Cn}{\sigma},
	\quad
	d_0 \leq \frac{C}{\sigma}, 
	\quad 
	\IP[A]\leq\frac{C}{n},
	\quad
	g_1,g_2,g_3\leq \frac{Cn^{3/4}}{\sigma},
	\quad
	d_1,d_2,d_3\leq \frac{C}{\sigma n^{1/4}};
}
this leads to
\be{
	s_1\leq \frac{Cn^{1/2}}{\sigma^2},
	\qquad
	s_2,s_3,s_4\leq \frac{Cn}{\sigma^3}.
}
Applying these bounds to Theorem~\ref{thm5} proves the claim.
\end{proof}

\subsubsection{Technical lemmas}

The following lemma is straightforward to prove. 

\begin{lma}\label{rem1} Consider the setting of Theorem~\ref{thm5}. Let~$A$ be an event such that~$\{G'\neq G^*\}\cup\{\Delta'\neq\Delta^{\!*}\} \subset A$. With
\bea{
	g_0^3 &= \IE\abs{G}^3,& 
	 g_1^4 &= \IE\clc{G^4\bone_A},&
	 g_2^4 &= \IE\clc{G'^4\bone_A},&
	 g_3^4 &= \IE\clc{G^{*4}\bone_A},\\
	d_0^3 &= \IE\abs{\Delta}^3, &
	 d_1^4 &= \IE\clc{\Delta^4\bone_A},&
	 d_2^4 &= \IE\clc{\Delta'^4\bone_A},&
	 d_3^4 &= \IE\clc{\Delta^{\!*4}\bone_A},
}
we have
\bg{
s_1  \leq \bcls{g_1g_2d_1d_2
		+ 2g_1g_2d_1d_3
		+ g_1g_3d_1d_3 }^{1/2},
		\qquad
s_2,s_3  \leq g_0d_0^2, \\
 s_4 \leq \bclr{g_1g_2 d_2
 		+ g_1g_2d_3 
 		+ g_1g_2(d_2\wedge d_3)
		+ g_1g_3(d_2\wedge d_3)}\IP[A]^{1/4}.
}
\end{lma}

\begin{lma}\label{lem3} Let~$Z\sim\Bi(r,q)$ with~$r\leq m$ and~$q\leq c/m$. Then for any~$a,b\geq 0$, 
\be{
	\IE\bclr{Z^a e^{bZ}}\leq C(a,b,c).
}
\end{lma}
\begin{proof} This easily follows from
\be{
	\IE\bclr{Z^a e^{bZ}} 
	 \leq \IE{e^{(a+b)Z}}
   = \exp\bclr{r\log(qe^{a+b}+1-q)} \leq \exp\bclr{rqe^{a+b}} 
   \leq \exp\bclr{ce^{a+b}}.\qedhere
}
\end{proof}

\begin{lma}\label{lem4} The family of random variables~$(\xi_1,\dots,\xi_n)$ is negatively associated; that is, for non-decreasing functions~$f$ and~$g$ and disjoint sets~$A,B\subset\{1,\dots,n\}$, 
\be{
	\IE\bclc{f\bclr{(\xi_i)_{i\in A}}g\bclr{(\xi_i)_{i\in B}}}
	\leq \IE{f\bclr{(\xi_i)_{i\in A}}}\IE{g\bclr{(\xi_i)_{i\in B}}}.
}
\end{lma}

}

\begin{lma}\label{lem5} Let~$a'$, $a^*$, $b$, $b'$ and~$b^*$ be real numbers. With~$\bar{x}=(x\wedge 1)\vee(- 1)$,
\bea{
	(i) \enskip& \babs{a'\bar{b}'-a^*\bar{b}^*}\\
	&\qquad\leq \abs{a'}(\abs{b'-b^*}\wedge 2)+\abs{a'-a^*}(\abs{b^*}\wedge 1){,}\\
	(ii)\enskip	& \babs{a'(\abs{b}\wedge\abs{b'}\wedge 1)\bone[bb'>0] - a^*(\abs{b}\wedge\abs{b^*}\wedge 1)\bone[bb^*>0] }\\
		& \qquad \leq \abs{a'}(\abs{b'-b^*}\wedge 1) + \abs{a'-a^*}(\abs{b}\wedge\abs{b^*}\wedge 1){,} \\
	(ii)\enskip	& \babs{a'(\abs{b}\wedge\abs{b'}\wedge 1)^2 \bone[bb'>0] - a^*(\abs{b}\wedge\abs{b^*}\wedge 1)^2\bone[bb^*>0] }\\
		& \qquad \leq 2\abs{a'}(\abs{b}\wedge 1)(\abs{b'-b^*}\wedge 1) + \abs{a'-a^*}(\abs{b}\wedge\abs{b'}\wedge 1)^2{.} \\
 }
\end{lma}
\begin{proof} 
We have
\bes{
	 \babs{a'\bar{b}'-a^*\bar{b}^*}
	&\leq \babs{a'\bar{b}'-a'\bar{b}^*} + \babs{a'\bar{b}^*-a^*\bar{b}^*}\leq \abs{a'}\babs{\bar{b}'-\bar{b}^*}+\abs{a'-a^*}\abs{\bar{b}^*} \\
	&\leq \abs{a'}(\abs{b'-b^*}\wedge 2)+\abs{a'-a^*}(\abs{b^*}\wedge 1),
}
which proves~$(i)$. Moreover, by considering eight possible combinations of the signs of~$b$, $b'$ and~$b^*$, we can conclude that
\ben{\label{112}
\abs{(\abs{b}\wedge\abs{b'}\wedge 1)\bone[bb'>0] - (\abs{b}\wedge\abs{b^*}\wedge 1)\bone[bb^*>0]}\le |b'-b^*|\wedge 1.
}
This in turn ensures that
\bes{
	& \abs{a'(\abs{b}\wedge\abs{b'}\wedge 1)\bone[bb'>0] - a^*(\abs{b}\wedge\abs{b^*}\wedge 1)\bone[bb^*>0] }\\
& \qquad\leq \abs{a'}\cdot\abs{(\abs{b}\wedge\abs{b'}\wedge 1)\bone[bb'>0] - (\abs{b}\wedge\abs{b^*}\wedge 1)\bone[bb^*>0] }\\
 &\qquad\qquad+\abs{a'-a^*}\cdot{(\abs{b}\wedge\abs{b^*}\wedge 1)\bone[bb^*>0] }\\
& \qquad\leq \abs{a'}(\abs{b'-b^*}\wedge 1)+\abs{a'-a^*}{(\abs{b}\wedge\abs{b^*}\wedge 1)},
}
as claimed in~$(ii)$. Lastly, using~$\abs{x^2 - y^2} = \abs{x+y}\cdot\abs{x-y}$ and~\eq{112} in the second inequality below, we obtain
\bes{
	& \abs{a'(\abs{b}\wedge\abs{b'}\wedge 1)^2\bone[bb'>0] - a^*(\abs{b}\wedge\abs{b^*}\wedge 1)^2\bone[bb^*>0] }\\
& \qquad\leq \abs{a'}\cdot\abs{(\abs{b}\wedge\abs{b'}\wedge 1)^2\bone[bb'>0] - (\abs{b}\wedge\abs{b^*}\wedge 1)^2\bone[bb^*>0] }\\
 &\qquad\qquad+\abs{a'-a^*}\cdot\abs{(\abs{b}\wedge\abs{b^*}\wedge 1)^2\bone[bb^*>0] }\\
& \qquad\leq \abs{a'}\bcls{(\abs{b}\wedge\abs{b'}\wedge 1)+(\abs{b}\wedge\abs{b^*}\wedge 1)}(\abs{b'-b^*}\wedge 1) 
\\
&\qquad\qquad+\abs{a'-a^*}{(\abs{b}\wedge\abs{b^*}\wedge 1)^2}\\
& \qquad\leq 2\abs{a'}(\abs{b}\wedge 1)(\abs{b'-b^*}\wedge 1) +\abs{a'-a^*}{(\abs{b}\wedge\abs{b^*}\wedge 1)^2},
}
which proves~$(iii)$.
\end{proof}


\section*{Acknowledgments}

This research was supported by the ARC Discovery Grants DP150101459 and DP190100613, the Singapore Ministry of Education Academic Research Fund Tier~2 Grant MOE2018-T2-2-076, and Singapore Ministry of Education Academic Research Fund Tier~1 Grants R-146-000-182-112 and R-146-000-230-114. We thank the Institute of Mathematical Sciences, NUS, for supporting the workshop \emph{Workshop on New Directions in Stein's Method} in March~2015, during which part of this research was conducted.

\setlength{\bibsep}{0.5ex}
\def\bibfont{\small}

\end{document}